\documentclass[12pt, final]{l4dc2024}

\usepackage{float}
\usepackage{todonotes}
\usepackage{algorithmicx}
\usepackage[vlined,ruled,linesnumbered]{algorithm2e}
\usepackage{graphicx}

\usepackage{comment}
\usepackage{siunitx}
\usepackage{relsize}
\usepackage{ifthen}

\usepackage{color}
\usepackage[T1]{fontenc}
\usepackage{psfrag}
\usepackage{booktabs}
\usepackage{mathtools}
\usepackage{xstring}
\usepackage{multirow}
\usepackage{xcolor}
\usepackage{prettyref}
\usepackage{flexisym}
\usepackage{bigdelim}
\usepackage{breqn} 
\usepackage{listings}

\usepackage{enumitem}
\usepackage{xspace}
\usepackage{bm}
\usepackage{tikz}
\usetikzlibrary{matrix,calc}

\newtheorem{assumption}{Assumption}
\newcommand{\cf}{\emph{cf.}\xspace}

\newcommand{\bdmath}{\begin{dmath}}
\newcommand{\edmath}{\end{dmath}}
\newcommand{\beq}{\begin{equation}}
\newcommand{\eeq}{\end{equation}}
\newcommand{\bdm}{\begin{displaymath}}
\newcommand{\edm}{\end{displaymath}}
\newcommand{\bea}{\begin{eqnarray}}
\newcommand{\eea}{\end{eqnarray}}
\newcommand{\beal}{\beq \begin{array}{ll}}
\newcommand{\eeal}{\end{array} \eeq}
\newcommand{\beas}{\begin{eqnarray*}}
\newcommand{\eeas}{\end{eqnarray*}}
\newcommand{\ba}{\begin{array}}
\newcommand{\ea}{\end{array}}
\newcommand{\bit}{\begin{itemize}}
\newcommand{\eit}{\end{itemize}}
\newcommand{\ben}{\begin{enumerate}}
\newcommand{\een}{\end{enumerate}}

\newcommand{\calC}{{\cal C}}

\newcommand{\calL}{{\cal L}}

\newcommand{\calS}{{\cal S}}
\newcommand{\calT}{{\cal T}}

\newcommand{\eg}{\emph{e.g.,}\xspace}
\newcommand{\ie}{\emph{i.e.,}\xspace}

\renewcommand{\boldsymbol}[1]{{\bm #1}}

\newcommand{\hide}[1]{}
\newcommand{\wrt}{w.r.t.\xspace}

\newcommand{\hiddenText}{{\color{gray} hidden text.}}
\newcommand{\hideWithText}[1]{\hiddenText}

\newcommand{\subject}{\text{ subject to }}
\DeclareMathOperator*{\argmax}{arg\,max}
\DeclareMathOperator*{\argmin}{arg\,min}

\newcommand{\tran}{^{\mathsf{T}}}

\newcommand{\Real}[1]{ { {\mathbb R}^{#1} } }

\newcommand{\bmat}{\left[ \begin{array}}
\newcommand{\emat}{\end{array}\right]}
\newcommand{\blue}[1]{{\color{blue}#1}}

\newcommand{\linkToPdf}[1]{\href{#1}{\blue{(pdf)}}}
\newcommand{\linkToPpt}[1]{\href{#1}{\blue{(ppt)}}}
\newcommand{\linkToCode}[1]{\href{#1}{\blue{(code)}}}
\newcommand{\linkToWeb}[1]{\href{#1}{\blue{(web)}}}
\newcommand{\linkToVideo}[1]{\href{#1}{\blue{(video)}}}
\newcommand{\linkToMedia}[1]{\href{#1}{\blue{(media)}}}
\newcommand{\award}[1]{\xspace} 

\newcommand{\xdot}{\dot{x}}

\newcommand{\thetadot}{\dot{\theta}}

\newcommand{\parentheses}[1]{\left(#1 \right)}

\newcommand{\bbN}{\mathbb{N}}
\newcommand{\bbZ}{\mathbb{Z}}

\newcommand{\controlset}{\mathbb{U}}
\newcommand{\xbr}{x_{\mathrm{B}}}
\newcommand{\xur}{x_{\mathrm{U}}}
\newcommand{\smoothregion}{\mathbf{O}}
\newcommand{\nonsmoothregion}{\mathbf{\Gamma}}

\newcommand{\clip}{\mathrm{clip}}
\newcommand{\Jnn}{J_{\mathrm{NN}}}
\newcommand{\lqr}{\mathrm{LQR}}
\newcommand{\supervise}{\mathrm{S}}

\newcommand{\hjb}{\mathrm{HJB}}
\newcommand{\Jvalue}{\mathrm{V}}
\newcommand{\smooth}{\mathrm{smooth}}
\newcommand{\pmp}{\mathrm{PMP}}
\newcommand{\local}{\mathcal{B}}

\usepackage{wrapfig}
\title[]{On the Nonsmooth Geometry and Neural Approximation of \\ the Optimal Value Function of Infinite-Horizon Pendulum Swing-up}
\usepackage{times}

\author{%
 \Name{Haoyu Han} \Email{hyhan@seas.harvard.edu}\\
 \addr School of Engineering and Applied Sciences, Harvard University%
 \AND
 \Name{Heng Yang} \Email{hankyang@seas.harvard.edu}\\
 \addr School of Engineering and Applied Sciences, Harvard University%
}

\begin{document}

\maketitle


\vspace{-6mm}

\begin{abstract}%
    We revisit the inverted pendulum problem with the goal of understanding and computing the true optimal value function. We start with an observation that the true optimal value function must be nonsmooth (\ie not globally $C^1$) due to the symmetry of the problem. We then give a result that can certify the optimality of a candidate \emph{piece-wise} $C^1$ value function. Further, for a candidate value function obtained via numerical approximation, we provide a bound of suboptimality based on its Hamilton-Jacobi-Bellman (HJB) equation residuals. Inspired by~\cite{holzhuter04automatica-optimal}, we then design an algorithm that solves backward the Pontryagin's minimum principle (PMP) ODE from terminal conditions provided by the locally optimal LQR value function. This numerical procedure leads to a piece-wise $C^1$ value function whose nonsmooth region contains periodic \emph{spiral lines} and smooth regions attain HJB residuals about $10^{-4}$, hence certified to be the optimal value function up to minor numerical inaccuracies. This optimal value function checks the power of optimality: (i) it sits above a polynomial lower bound; (ii) its induced controller globally swings up and stabilizes the pendulum, and (iii) attains lower trajectory cost than baseline methods such as energy shaping, model predictive control (MPC), and proximal policy optimization (with MPC attaining almost the same cost). We conclude by distilling the optimal value function into a simple neural network. Our code is avilable in https://github.com/ComputationalRobotics/InvertedPendulumOptimalValue.
\end{abstract}


\section{Introduction}
\label{sec:introduction}

Inverted pendulum is arguably one of the most fundamental problems in nonlinear (optimal) control.\begin{wrapfigure}{r}{0.25\textwidth}
	\includegraphics[width = 0.25\textwidth]{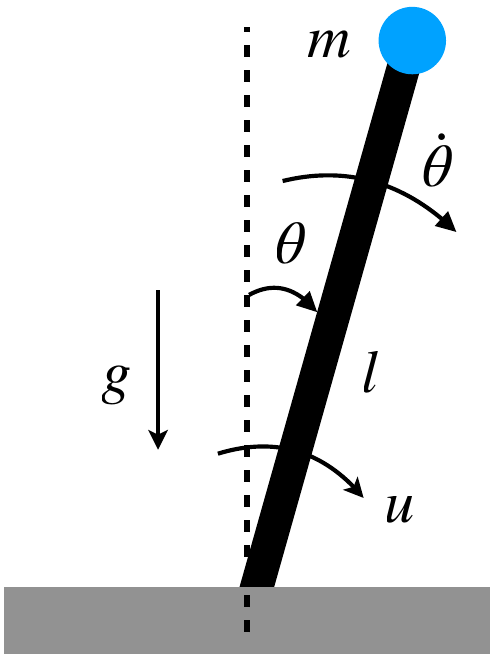}
	\vspace{-9mm}
	\caption{Pendulum.\label{fig:pendulum}}
	\vspace{-4mm}
\end{wrapfigure} 
It has been frequently used in textbooks~\citep{sontag13book-mathematical,slotine91book-applied,tedrake09book-underactuated, Khalil02book-nonlinear} to illustrate foundational concepts such as feedback linearization, Lyapunov stability, proportional-integral-derivative (PID) control, energy shaping, to name a few. More recently, inverted pendulum is also one of the most basic benchmark problems for reinforcement learning, \eg in the Deepmind control suite~\citep{tassa18-deepmind}. Not only is the inverted pendulum a theoretically interesting problem to study, it also relates to practical applications in model-based humanoid control~\citep{feng14-optimization,sugihara02-real}.

One can often consider the inverted pendulum as a solved nonlinear control problem because in the model-based paradigm there exists elegant solutions such as energy pumping plus local linear-quadratic-regulator (LQR) stabilization~\citep{aastrom00automatic-swinging,muskinja06tie-swinging}; and in the model-free paradigm algorithms such as proximal policy optimization (PPO) and actor critic work very well~\citep{raffin21jmlr-stablebaseline,ren23arxiv-stochastic}. However, from the perspective of \emph{optimal control}, we know very little about the true optimal value function (or cost-to-go) and its associated optimal controller.
This leads to the side effect that we cannot evaluate the suboptimality of other (approximately optimal) controllers. Let us state the continuous-time infinite-horizon (undiscounted) pendulum swing-up problem to understand why it is challenging to compute the optimal controller and value function.

{\bf Problem Setup}. 
We are given the continuous-time pendulum dynamics as shown in Fig.~\ref{fig:pendulum}
\bea \label{eq:pendulum-dynamics}
x := \begin{bmatrix}
	\theta \\ \thetadot
\end{bmatrix}, 
\quad 
\xdot(t)  = f(x(t),u(t)) := 
\begin{bmatrix}
\thetadot  \\
-\frac{1}{ml^{2}} \left( b \thetadot - m g l \sin \theta - u \right)
\end{bmatrix},
\eea
where $\theta$ is the angular position, $\thetadot$ is the angular velocity, $m$ is the point mass, $l$ is the length of the pole, $b$ is the damping coefficient, $g$ is the gravity constant, and $u$ is the torque. Our goal is to swing up and stabilize the pendulum from any initial state $x_0$ to the upright position $\xur = [2k \pi,0]\tran, \forall k \in \bbZ$, an unstable equilibrium point. We formulate the undiscounted optimal control problem
\begin{equation}\label{pro:optimal-control-problem} 
J^*(x_0) = \min _{u(t)} \displaystyle \int ^{+\infty }_{0}c(x(t),u(t)) dt, \quad \subject \ \ x(0) = x_0, \ \ u(t) \in \controlset, \ \ \text{and }~\eqref{eq:pendulum-dynamics}
\end{equation}
where the cost function $c(x,u)$ is defined as
\bea\label{eq:cost-function}
c( x, u ) = q_1 \sin^2 \theta + q_1 (\cos \theta -1 )^2 + q_2 \thetadot^2+  r u^2,
\eea
with $q_1,q_2,r > 0$. We let $\controlset$ in~\eqref{pro:optimal-control-problem} be either $\Real{}$ (without control saturation) or $[-u_\max, u_\max]$ with $u_\max < mgl$ (with control saturation). Note that we use ``$\sin \theta$'' and ``$\cos \theta$'', instead of $\theta$, in the cost function~\eqref{eq:cost-function} to avoid the modulo $2\pi$ issue. It is not difficult to observe that $J^*(\xur) = 0$ and $J^*(x_0)$ is positive definite because swinging up from any $x_0$ that is not $\xur$ would incur a strictly positive total cost. Problem~\eqref{pro:optimal-control-problem} is a nonlinear quadratic regulator problem~\citep{wernli75automatica-suboptimal}.

We make an assumption about the set of admissible control trajectories.

\begin{assumption}[Admissible Control]\label{assum:admissible-control}
In problem~\eqref{pro:optimal-control-problem}, the control sequence $u(t)$ is admissible if (i) $u(t)$ is piece-wise continuous, and (ii) $x(t) \rightarrow \xur$ under $u(t)$ when $t \rightarrow +\infty$.
\end{assumption}

Intuitively, condition (ii) in Assumption~\ref{assum:admissible-control} allows us to only consider the set of controllers that asymptotically stabilize the pendulum at $\xur$. When $b$ is not too large, energy shaping followed by local LQR is such an admissible controller (hence the admissible control set is nonempty).

\vspace{-3mm}
\subsection{Related Work}
\vspace{-1mm}

{\bf Dynamic Programming}. A straightforward approach for solving~\eqref{pro:optimal-control-problem} is to discretize the dynamics~\eqref{eq:pendulum-dynamics} and perform value iteration with barycentric interpolation~\citep{Munos1998neuIPS-Barycentric}. Not only will this approach suffer from the curse of dimensionality, it is also unclear whether it will converge in the undiscounted case, as shown in~\cite[Example 2.3]{yang23book-optimal}. 

{\bf Hamilton-Jacobi-Bellman (HJB) Equation}. The HJB theorem~\citep[Theorem 7.1]{tedrake09book-underactuated}~\citep{kamalapurkar2018book-sufficienttheorem} states that if one can find a $C^1$ function $J(x)$ such that $J(\xur)=0$, $J(x)$ is positive definite and satisfies the HJB equation
\bea\label{eq:HJB}
\min_{u \in \controlset} {c(x,u) +  \frac{\partial J}{\partial x}\tran f(x,u) = 0}, \quad \forall x
\eea
then $J(x)$ is the optimal value function. Obtaining an analytic solution to~\eqref{eq:HJB} is often impossible, hence numerical approximations are needed. The levelset algorithm~\citep{mitchell2005-levelset,osher1988jcp-levelset,osher2001jcp-levelset} is a popular method to solve Hamilton-Jacobi (HJ)-type equations, in particular those appearing in reachability problems~\citep{bansal17cdc-hamilton}. Nevertheless, to the best of our knowledge, it is not yet applicable to the pendulum problem because~\eqref{eq:HJB} cannot be transformed into an HJB equation that has a time derivative and terminal condition. A fundamental problem of the HJB equation~\eqref{eq:HJB} is that it implicitly assumes the optimal value function is $C^1$, which is not true for the pendulum problem, as we will show in Theorem~\ref{the:non-smooth}. One can consider the notion of a \emph{viscosity solution}~\citep{bardi97book-optimal} to avoid this issue, but it does not make the computation any easier. A family of finite-element methods~\citep{jensen13sina-convergence,smears14sina-discontinuous,kawecki22fcm-convergence} considers the stochastic optimal control problem where \eqref{eq:HJB} becomes an elliptic PDE. However, they do not consider the infinite-horizon case where a boundary condition is unavailable.

{\bf Pontryagin's Minimum Principle (PMP)}. Another classical result in optimal control is PMP (to be reviewed in Lemma~\ref{lemma:PMP})~\citep{bertsekas12book-dp}, which states the optimal state-control trajectory must satisfy an ODE (but trajectories satisfying the ODE may not be optimal).
\citep{holzhuter04automatica-optimal,hauser2001acc-geometry} uses the local LQR value function of the pendulum to provide boundary conditions for PMP and computes a value function that swings up the pendulum. However, they only considered the case of no control saturation and did not prove optimality of the value function.

{\bf Weak Solution}. Due to the difficulty of computing and certifying the optimal value function, \cite{Lasserre2007SIAM-lowerbound,lasserre05cdc-nonlinear} developed a general framework of using convex relaxations to compute smooth \emph{weak solutions} of the HJB~\eqref{eq:HJB}~\citep{vinter93sicopt-convex}. \cite{Yang23ral-lowerbound} recently applied this method to compute polynomial lower bounds of the optimal value function. However, because the true optimal value function is nonsmooth, polynomial approximation is not expected to capture the detailed geometry of the optimal value function, as we will show in Fig.~\ref{fig:compare}. 


{\bf Neural Approximation}. In addition to the aforementioned classical methods, using neural networks to approximate the optimal value function becomes increasingly popular~\citep{lutter2020crl-neural-approximation,shilova2023-neural-approximation}. 
\citep{doya2000neuralcomp-RL-HJBresidual,munos1999-HJBresidual} first introduced HJB residual, \ie violation of~\eqref{eq:HJB}, as a loss to train neural networks~\citep{raissi19jcp-physics}, followed by~\cite{tassa2007tonn-neural-approximation} showing how to avoid local minima, and~\cite{liu2014toc-neural-approximation} showing how to make it robust to dynamic disturbance. However, the problem remains that only using HJB loss may lead to multiple solutions. Another line of work uses PMP to generate data for training~\citep{nakamura2021siam-PMP-plus-NN}, but it requires solving a boundary value problem which may also have multiple solutions. In general, neural approximation also faces the same difficulty that the optimal value function may be nonsmooth, and it remains difficult to evaluate its suboptimality.

\vspace{-3mm}
\subsection{Contributions}
\vspace{-1mm}


We start with an observation (Theorem~\ref{the:non-smooth}) that the optimal value function $J^*(x)$ of~\eqref{pro:optimal-control-problem} must be nonsmooth at the bottomright position due to symmetry of the problem, and hence the HJB equation~\eqref{eq:HJB} cannot be satisfied everywhere in the state space. In such cases, little is known about $J^*(x)$ except that it is the so-called \emph{viscosity solution} of the HJB~\citep{bardi97book-optimal}, which is difficult to interpret for practitioners. We contribute a result that is easy to interpret (Theorem~\ref{the:optimality}), using elementary proof, that can certify the optimality of a given candidate \emph{piece-wise $C^1$} function. For numerically computed approximately optimal value functions, we give a result (Theorem~\ref{the:suboptimality}) that certifies the \emph{suboptimality} of the numerical solution \wrt the true optimal value function.

We then develop a numerical approach that, for the first time, computes the true optimal value function of pendulum swing-up, up to minor numerical inaccuracies. Our algorithm is inspired by the algorithm of~\cite{holzhuter04automatica-optimal} and is based on PMP with boundary conditions provided by local LQR, but it makes several improvements. For example, we handle the case with control saturation, we uncover a nonsmooth curve in the optimal value function, and we can bound the suboptimality of our solution using Theorem~\ref{the:suboptimality}. We then showcase the power of optimality. (a) The controller induced from the optimal value function swings up and stabilizes the pendulum from any initial state. (b) The induced controller achieves \emph{lower} cost than existing controllers such as energy pumping, reinforcement learning, and model predictive control (MPC), with the MPC controller being the best baseline as it achieves almost the same cost as our controller. (c) The optimal value function indeed sits above the polynomial lower bound obtained from convex relaxations.

Our numerical algorithm is expensive as it requires solving a large amount of PMP trajectories, computing intersections, and storing dense samples of the optimal value function. We therefore ask if we can use a neural network to \emph{distill} and \emph{compress} the optimal value function. In the supervised case, we show that we just need $50$ optimal value samples to train a simple neural network whose induced controller can globally swing up the pendulum. In the weakly supervised case, we design a novel loss function to train a neural network directly from \emph{raw PMP trajectories}, and the resulting controller still globally swings up the pendulum. This simple training scheme generalizes to the more challenging cart-pole problem, where we also obtain a global stabilizing controller.

{\bf Limitations}. Unfortunately, there are still puzzles related to the true optimal value function (in our opinion, due to the limitations of fundamental theoretical tools in optimal control). In the case with control saturation, we observe and conjecture that the optimal value function is \emph{discontinuous}. Although we cannot formally prove our conjecture, we provide numerical evidence based on the limiting discounted viscosity solution idea in~\cite{bardi97book-optimal}. 




\section{Certificate of (Sub-)Optimality for the Nonsmooth Value Function}

We start with an observation that the optimal value function $J^*(x)$ of~\eqref{pro:optimal-control-problem} must be nonsmooth.

\begin{theorem}[Nonsmooth Optimal Value Function]\label{the:non-smooth}
    The optimal value function $J(x)$ to problem \eqref{pro:optimal-control-problem} is not $C^1$ at the bottomright position $\xbr := [\pi + 2k\pi,0]\tran,\forall k \in \bbZ$.
\end{theorem}

The proof of Theorem~\ref{the:non-smooth} is given in Appendix~\ref{sec:app:proof-nonsmoothnes}. Here we provide a brief explanation.
If $J(x)$ were smooth at $\xbr$, then it must satisfy the HJB equation~\eqref{eq:HJB}, implying the optimal controller at $\xbr$ must be unique due to strong convexity of the cost~\eqref{eq:cost-function}. However, our physics insight tells us swinging up the pendulum from the left side is equivalent to swinging up from the right side (achieving the same cost), leading to two symmetric optimal controllers, thus a contradiction.

\vspace{-3mm}
\subsection{Certificate of Optimality}
\vspace{-1mm}

We then state a result that verifies the optimality of a candidate piece-wise $C^1$ value function for~\eqref{pro:optimal-control-problem}.

\begin{theorem}[Optimality Certificate of A Piece-wise $C^1$ Value Function] \label{the:optimality}
    Let $\smoothregion_{-N},\dots,\smoothregion_N$ be open subsets of $\Real{2}$ that satisfy
    \begin{enumerate}[label=(\roman*)]
        \item $\cup^N_{i=-N}\smoothregion_i = \Real{2}$,
        \item $\forall i$, $\smoothregion_i \cap \smoothregion_{i+j} \neq  \emptyset$ if $j = \pm 1$, and $\smoothregion_i \cap \smoothregion_{i+j} =  \emptyset$ if $|j| > 1$,
    \end{enumerate}
    and $J_{-N},\dots,J_N(x)$ be $C^1$ functions defined on them, respectively ($N$ possibly infinite).
    Define 
    $$
    J(x) = \min_i\{J_i(x)|x\in \smoothregion_i\}.
    $$ 
    If $J(x)$, $\smoothregion_i$'s, and $J_i(x)$'s are such that
    \begin{enumerate}[label=(\roman*)]
        \setcounter{enumi}{2}
        \item $J(x)$ is continuous and piece-wise $C^1$ on $\Real{2}$,
        \item $J(\xur) = 0$ where $\xur = [2k\pi,0]\tran,\forall k \in \bbZ$ is the upright position,
        \item $\forall i$, $J_i(x)$ satisfies the HJB equation~\eqref{eq:HJB} everywhere on $\smoothregion_i$,
        \item\label{cond:monotonic} the nonsmooth curve $\nonsmoothregion:= \{x\in \Real{2} | \exists(i,j) \text{ s.t. } J_i(x) = J_j(x)\}$ can be locally defined by $\{x|G(x) = 0\}$ with $G$ a $C^1$ function, and every admissible trajectory $x(t)$ satisfies $G(x(t))$ is monotonic in $t$ near an intersection point $x(t_0)$ where $G(x(t_0)) = 0$,
        \item $\forall x_0 \in \Real{2}$, there exists a trajectory $(x(t),u(t))$ starting from $x_0$ that attains cost $J(x_0)$,
    \end{enumerate}
    then $J(x)$ is the optimal value function of~\eqref{pro:optimal-control-problem}.\footnote{If $J(x)$ is discontinuous, we require admissible trajectories to not cross the discontinuous region to attain lower costs. See details in Appendix~\ref{app:sec:bang-bang-discontinuous}.}
    \end{theorem}

The proof of Theorem~\ref{the:optimality} is provided in Appendix~\ref{sec:app:proof-optimality}. Theorem~\ref{the:optimality} provides a list of conditions to certify optimality of a piece-wise $C^1$ function $J(x)$. The only technical condition that is difficult to verify is \ref{cond:monotonic}, which is necessary to avoid state trajectories that cross the nonsmooth region $\nonsmoothregion$ in a pathological way, \eg imagine $\sin(\frac{1}{t})$ crossing the $x$-axis when $t$ tends to $0$. 

In the pendulum problem, each $\smoothregion_i$ is an open set containing $\xur$ and differs from $\smoothregion_{i\pm 1}$ by a shift of $2\pi$ along the $\theta$-axis, with $J_i(x)$ defined on it($J_i(x)$ is equal to $J_{i\pm 1}(x)$ by shifting $2\pi$). $\nonsmoothregion$ composes of an infinite number of nonsmooth \emph{spiral} lines, again shifted by $2\pi$ along the $\theta$-axis, intersected by $J_i(x)$ and $J_{i \pm 1}(x)$. The numerical algorithm we develop in Section~\ref{sec:approach}, based on PMP, ensures each $J_i(x)$ satisfies HJB~\eqref{eq:HJB} on $\smoothregion_i$\footnote{
    The satisfaction of HJB and $C^1$ is not entirely precise as it relies on numerical solutions, hence the development of Theorem~\ref{the:suboptimality} for error estimation.}, $J_i(x)$ is $C^1$, and $J(x)$ is attainable. For more details please refer to Figure \ref{fig:pendulum-value}.

\vspace{-2mm}
\subsection{Certificate of Suboptimality}
\vspace{-1mm}

Finding analytical solutions that exactly satisfy Theorem~\ref{the:optimality} is intractable. For numerically computed candidate value functions, we wish to compute a suboptimality certificate \wrt $J^*(x)$ of~\eqref{pro:optimal-control-problem}. Toward this, we need to first review the local LQR controller of the inverted pendulum.

{\bf Local LQR}. The pendulum dynamics~\eqref{eq:pendulum-dynamics} satisfies $f(\xur,0) = 0$ and we can linearize $f(x,u)$ around $(\xur,0)$ to obtain a linear system 
\bea\label{eq:linearization}
\xdot = A (x-\xur) + Bu, \quad A = \frac{\partial f}{\partial x}(\xur,0), B = \frac{\partial f}{\partial u}(\xur,0).
\eea
Similarly, we can perform a quadratic approximation of the cost function $c(x,u)$ around $(\xur,0)$
\bea\label{eq:cost-quadratic}
c(x,u) \approx q_1 (\theta - 2k\pi)^2 + q_2 \thetadot^2 + ru^2 = (x - \xur)\tran Q (x - \xur) + r u^2.
\eea
The optimal value function for minimizing~\eqref{eq:cost-quadratic} subject to~\eqref{eq:linearization} is a quadratic function
\bea \label{eq:LQR-value}
J_\infty(x) = (x- \xur)\tran P (x-\xur),
\eea
where $P \succ 0$ is the unique positive definite solution to the algebraic Riccati equation
$$
A\tran P + P A - \frac{1}{r} P B B\tran P + Q = 0.
$$

We now introduce a suboptimality certificate for any candidate $C^1$ value function.

\begin{theorem}[Sub-Optimality Certificate of A $C^1$ Value Function]
    Let $\calL := \{ x \in \Real{2} \mid J_\infty(x) \leq \varepsilon \}$ be defined that $\calL$ is a region of attraction using local LQR controller. Let $T_x > 0$ be the time taken by the optimal controller to enter region $\calL$ from initial state $x \in \Real{2}$. If $J(x)$ is a $C^1$ function satisfy certain conditions and there exists a continuous function $l(x)$ s.t.
        \bea\label{eq:subop}
        \min_{u \in \controlset} c(x,u) + \frac{\partial J}{\partial x}\tran f(x,u) = l(x), \quad \Vert l(x) \Vert  < \epsilon, \quad \forall x \in \Real{2},
        \eea
then $J(x)$ has bounded error from $J^*(x)$ as
\bea
J^*(x) \leq J(x) \leq J^*(x)+\epsilon T_x  + \delta +\varepsilon.
\eea
\end{theorem}

\begin{theorem}[Sub-Optimality Certificate of A $C^1$ Value Function]\label{the:suboptimality}
    Let $\calL := \{ x \in \Real{2} \mid J_\infty(x) \leq \varepsilon \}$ be defined with a sufficiently small $\varepsilon > 0$ such that $\calL$ is a region of attraction for $\xur$ using the local LQR controller within the control bounds $\controlset$. Let $T_x > 0$ be the time taken by the optimal controller to enter region $\calL$ from initial state $x \in \Real{2}$. If $J(x)$ is a $C^1$ function on $\Real{2}$ that satisfies
    \begin{enumerate}[label=(\roman*)]
        \item\label{cond:subopt:one} $| J(x) - J_\infty(x) | \leq \delta$ for any $x \in \calL$, and
        \item\label{cond:subopt:two} there exists a continuous function $l(x)$ such that
        \bea\label{eq:subop}
        \min_{u \in \controlset} c(x,u) + \frac{\partial J}{\partial x}\tran f(x,u) = l(x), \quad \Vert l(x) \Vert  < \epsilon, \quad \forall x \in \Real{2},
        \eea
        \item $\forall x_0 \in \Real{2}$, there exists a trajectory $(x(t),u(t))$ starting from $x_0$ that attains cost $J(x_0)$,
    \end{enumerate} 
then $J(x)$ has bounded error from $J^*(x)$ as
\bea
J^*(x) \leq J(x) \leq J^*(x)+\epsilon T_x  + \delta +\varepsilon.
\eea
\end{theorem}

The proof of Theorem~\ref{the:suboptimality} is given in Appendix~\ref{sec:app:proof-suboptimality}.
Theorem~\ref{the:suboptimality} is computationally useful as $J(x)$ is usually a $C^1$ function interpolated from samples. Condition~\ref{cond:subopt:one} is easy to realize, in fact, one can choose $J(x) \equiv J_\infty(x)$ for $x \in \calL$ so that $\delta = 0$ (as what we will do in Section~\ref{sec:approach}, we will solve backward ODEs from $J_\infty(x)$ to get $J(x)$, so in $\calL$ they are the same). Condition~\ref{cond:subopt:two} is also checkable as one can compute $l(x)$ from $J(x)$ (the minimization in~\eqref{eq:subop} is closed-form solvable) and evaluate $\epsilon$. $T_x$ needs to be estimated. In practice, we approximate $T_x < 10$ as we can swing up the pendulum to region $\calL$ within ten seconds.\footnote{Or we can approximate $T_x$ by $J_{\epsilon_1}-J_{\epsilon_2} \approx (\epsilon_1-\epsilon_2)T_x$.}


\vspace{-4mm}
\section{Numerical Approximation by Pontryagin's Minimum Principle}
\label{sec:approach}

We design an algorithm based on PMP to compute a value function that verifies Theorem~\ref{the:optimality}-\ref{the:suboptimality}.

\vspace{-4mm}
\subsection{Numerical Procedure}
\vspace{-1mm}

We begin by recalling Pontryagin's minimum principle, which can be derived using the method of characteristics for the HJB~\eqref{eq:HJB}.

\begin{lemma}[Pontryagin's Minimum Principle]\label{lemma:PMP}
    Let $(u^*(t),x^*(t)),t\in[0,T] $ be a pair of optimal control and state trajectories satisfying dynamics~\eqref{eq:pendulum-dynamics} and $x^*(0) = x_0$ as given. Let $p(t)$ be the solution of the adjoint equation \emph{almost everywhere}
    \bea\label{eq:adjoint}
    \dot{p}(t) = -\nabla_x H(x^*(t),u^*(t),p(t)),\quad p(T) = \nabla_x J(x^*(T))
    \eea
    where $J$ is the optimal value function and $H$ is the Hamiltonian defined by
\bea
    H(x,u,p) = c(x,u) + p^T f(x,u)
\eea
Then, for almost every $t \in [0, T]$ we have
\bea \label{eq:pmp:solve-u}
u^*(t) = \argmin_{u\in \controlset } H(x^*(t),u,p(t))
\eea
\end{lemma}

To use Lemma~\ref{lemma:PMP}, we will (i) solve the problem \eqref{eq:pmp:solve-u}, and (ii) provide a terminal condition $p(T)$.

{\bf Solve $u^*$}. Observe that the pendulum dyamics~\eqref{eq:pendulum-dynamics} is control-affine 
$$
f(x,u) = f_1(x) + f_2(x)u, \quad f_1(x) = \begin{bmatrix}
    \thetadot \\ - \frac{1}{ml^2} (b \thetadot - mgl \sin\theta)
\end{bmatrix}, f_2(x) = \begin{bmatrix}
    0 \\ \frac{1}{ml^2}
\end{bmatrix},
$$
and the cost function $c(x,u)$~\eqref{eq:cost-function} is quadratic in $u$. Therefore, the solution to~\eqref{eq:pmp:solve-u} is
\bea \label{eq:solution-u}
u^* = \begin{cases}
    - \frac{1}{2r} p\tran f_2(x) & \text{if } \controlset = \Real{} \\
    \clip \parentheses{ - \frac{1}{2r} p\tran f_2(x), - u_\max, u_\max} & \text{if } \controlset = [-u_\max, u_\max]
\end{cases}
\eea 
where the ``$\clip$'' function saturates the control between $- u_\max$ and $u_\max$. Inserting~\eqref{eq:solution-u} back to the adjoint equation~\eqref{eq:adjoint} and the original dynamics~\eqref{eq:pendulum-dynamics}, we obtain an ODE in the optimal state $x^*(t)$ and the co-state $p(t)$, which can be solved when boundary conditions are provided.

{\bf Terminal Condition}. Inspired by~\cite{holzhuter04automatica-optimal}, we provide terminal conditions of the PDE, \ie a pair of $x^*(T_x)$ and $p(T_x)$ (because the associated $p(0)$ with $x^*(0)$ is unavailable). Because the LQR value function~\eqref{eq:LQR-value} is locally optimal around $\xur$, for any $x^*(T_x)$ that is on the boundary of the small ellipse (such that $\calL$ is defined as in Theorem~\ref{the:suboptimality})
\bea\label{eq:boundary-ellipse}
\partial \calL = \{ x \in \Real{2} \mid J_\infty(x) = \varepsilon \},
\eea
we can approximate 
$$
p(T_x) = 2P(x^*(T_x) - \xur).
$$
Once $x^*(T_x)$ and $p(T_x)$ are available, we can solve the ODE using \emph{backward integration} to obtain a locally optimal trajectory that satisfies PMP.

{\bf Sample $x^*(T_x)$}. We then wish to densely sample $x^*(T_x)$ on $\partial \calL$~\eqref{eq:boundary-ellipse} to obtain a large amount of PMP trajectories to densely cover the state space $\Real{2}$. A naive uniform sampling strategy will lead to trajectories clustered in certain regions and do not fully cover $\Real{2}$. Inspired by~\cite{holzhuter04automatica-optimal}, we sample $x^*(T_x)$ based on a distance metric between two PMP trajectories. Let $x_1^*(t)$ and $x_2^*(t)$ be two PMP trajectories already computed, the distance between these two trajectories is defined as 
\bea\label{eq:distance-metric}
d(x_1^*(t),x_2^*(t)) = \Vert x_1^c - x_2^c \Vert, \quad x_i^c = \{ x_i(t_c) \mid J(x_i(t_c)) = V_c, t_c \in [0,+\infty] \},i=1,2,
\eea
with $V_c$ a positive number larger than $\varepsilon$ (\eg $V_c = 10000\varepsilon$). The idea of this metric is to ensure the trajectories stay close after backward integration. The sampling algorithm is designed to make adjacent PMP trajectories have equal distances based on~\eqref{eq:distance-metric}. Details are provided in Appendix~\ref{sec:app:sampling}.

\setlength{\textfloatsep}{0pt}%
\begin{algorithm}[h]
    \SetAlgoLined
    \caption{Compute the Nonsmooth Curve \label{alg:nonsmooth-line}}
    \textbf{Input:} PMP trajectories $\calT$; small value increment $\Delta>0$; number of values $M$\\
    \textbf{Output:} Set of intersection points $\calS$\\
    \For{$k \gets 1$ to $M$}{
        $\calC \leftarrow \emptyset$ \label{line:emptyset}\\
        \For{$\tau$ in $\calT$}{
            $j_\max = \max \{ j \mid \tau(j).\mathrm{value} \leq k \Delta \}$\\
            $\calC  \leftarrow \calC \cup  \tau(j_\max).\mathrm{state}$
        } \label{line:computestates}
        $S = \mathrm{shift\_intersect}(\calC)$ \label{line:intersect}\\
        $\calS \leftarrow \calS \cup S$
    }
\end{algorithm}

{\bf Intersection of PMP Trajectories \& the Nonsmooth Curve}. After we obtain a large set of PMP trajectories (\cf Appendix~\ref{app:sec:pmp-trajectories}), they will intersect with each other and themselves on a 2D plane. We calculate the state where a trajectory intersect with others at the first time, in order to stop it there. All these terminal states form a spiral line, which you can find in the middle column of Figure \ref{fig:pendulum-value}. 

Given two PMP trajectories $x^*_1(t)$ and $x^*_2(t)$, if there exist $t_1$ and $t_2$ such that $x^*_{12} = x^*_1(t_1) = x^*_2(t_2)$ and $J_1(x^*_1(t_1)) = J_2(x^*_2(t_2))$ (here $J_1$ and $J_2$ are the same as in Theorem~\ref{the:non-smooth}), then $x^*_{12}$ is an intersection point, from which there exist (at least) two optimal\footnote{Here "optimal" solely means it satisfied PMP and is a candidate for optimal trajectory} trajectories achieving the same cost. Therefore, by the same reasoning as in Theorem~\ref{the:non-smooth}, $x^*_{12}$ can be a point at which the optimal value function is nonsmooth. Algorithm~\ref{alg:nonsmooth-line} presents a method to compute all these intersection points. Given a set of PMP trajectories $\calT$ where each trajectory $\tau \in \calT$ contains a sequence of states and values (\ie $x^*(t)$ and $J(x^*(t))$ at discrete timesteps, $\tau(j).\mathrm{value}$ and $\tau(j).\mathrm{state}$ represent the value and state, respectively), line~\ref{line:emptyset}-\ref{line:computestates} computes all the states of the trajectories that have value $k \Delta$. Among these states $\calC$, line~\ref{line:intersect} finds the common states by first forming a polygon using the points in $\calC$ and then intersect $\calC$ with a copy of $\calC$ shifted along $\theta$-axis by $2\pi$\footnote{One can treat $\calC$ as contour line}. The output $\calS$ thus contains all such intersection points forming a spiral line.

{\bf Controller Synthesis}. 
After getting the nonsmooth curves, we restrict all raw PMP trajectories to lie inside the nonsmooth curves. Then, we interpolate the value samples to obtain the value function. To synthesize controls, we use the solution in~\eqref{eq:solution-u} with interpolated co-state $p$ from samples.

\vspace{-3mm}
\subsection{Results}
\label{sec:approach:results}

{\bf Setup}. 
We use
$m = 1$, $b = 0.1$, $l = 1$, $g = 9.8$, $q_1 = 1$, $q_2 = 1$, $r=1$ in the dynamics~\eqref{eq:pendulum-dynamics} and cost function~\eqref{eq:cost-function}. We set $u_\max = 2$ in the case of control saturation. We are interested in the optimal value function on the region $x \in [-8,8]\times[-8,8]$, as it contains $[0,0],[2\pi,0]$ and $[-2\pi,0]$ (once we obtain $J$ on this region we can shift it by $2k\pi$ to get other regions). We set $\varepsilon=0.0002$ in~\eqref{eq:boundary-ellipse}.

{\bf Optimal Value Function}. Fig.~\ref{fig:pendulum-value} shows the optimal value functions both (a) without control saturation and (b) with control saturation. The middle column of Fig.~\ref{fig:pendulum-value} draws the nonsmooth curves obtained using Algorithm~\ref{alg:nonsmooth-line}, with the colored regions indicating the regions of attraction to the upright position $\xur$ (\eg for any initial state in the blue region, the optimal trajectory will stay in the blue region and converges to $\xur$). In each of the colored regions, the HJB residuals, \ie $l(x)$ in Theorem~\ref{the:suboptimality}, are about $10^{-4}$. Therefore, according to Theorems~\ref{the:optimality} and~\ref{the:suboptimality}, we can conclude the numerically computed value functions in Fig.~\ref{fig:pendulum-value} are the optimal value functions, up to minor numerical inaccuracies and suboptimality. To further verify the correctness of the optimal value functions, Fig.~\ref{fig:compare} compares the numerical value function with a smooth degree-7 polynomial lower bound computed using SOS relaxations in the case of control saturation~\citep{Yang23ral-lowerbound}. As we can see, the optimal value function sits above the polynomial lower bound, and the smooth polynomial hardly captures the nonsmooth geometry, especially around $[\pm2\pi,0]$. 

\begin{remark}[Discontinuity]
    The optimal value function in Fig.~\ref{fig:pendulum-value}(b) appears to be \emph{discontinuous}. This is a puzzle that we cannot formally (dis-)prove. Even after adding a discount factor in the cost~\eqref{eq:cost-function}, the discontinuous phenomenon remains, see Appendix~\ref{app:sec:discount-factor}. As a result, we cannot conclude the (dis-)continuity of the true optimal value function by using \cite[Theorem 1.5]{bardi97book-optimal}.
\end{remark}

\begin{figure}[t]
	\vspace{-14mm}
	\begin{center}
		\begin{tabular}{ccc}
            \hspace{-5mm}	
            \begin{minipage}{0.42\textwidth}
				\centering
				\includegraphics[width=\textwidth]{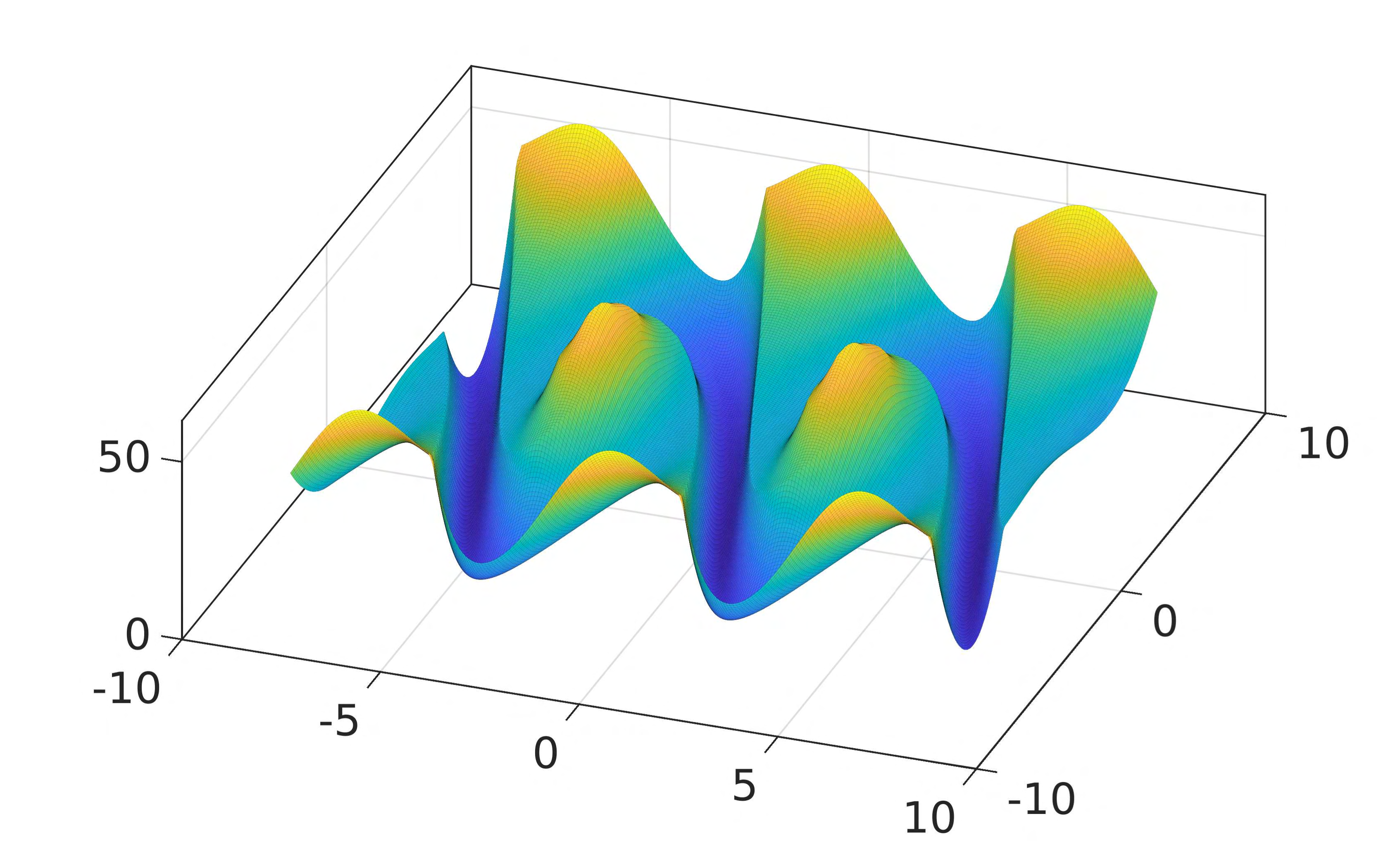}
			\end{minipage}
			&\hspace{-6mm}
			\begin{minipage}{0.28\textwidth}
				\centering
				\includegraphics[width=\textwidth]{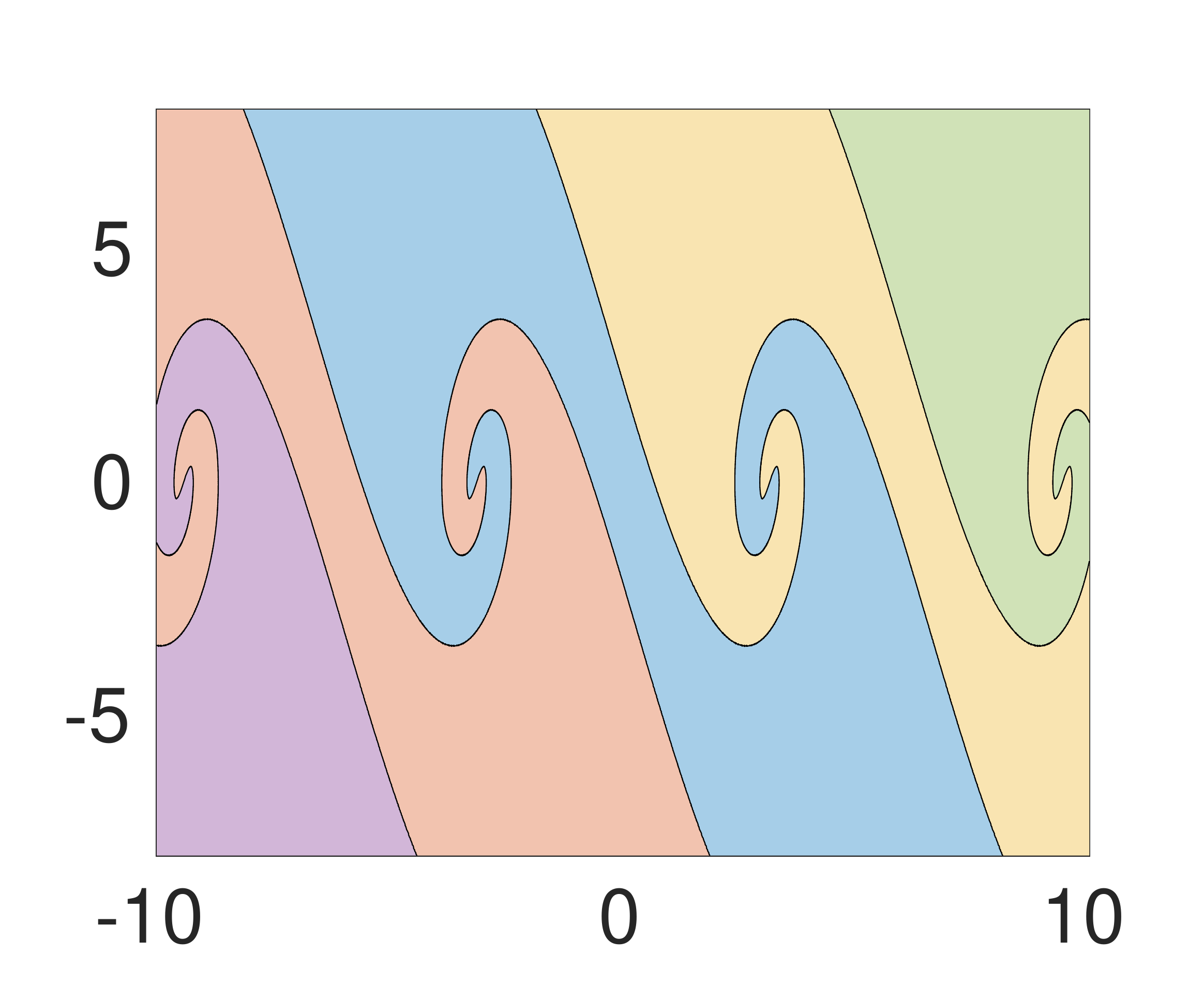}
			\end{minipage}
            &
            \hspace{-8mm}
			\begin{minipage}{0.37\textwidth}
				\centering
				\includegraphics[width=\textwidth]{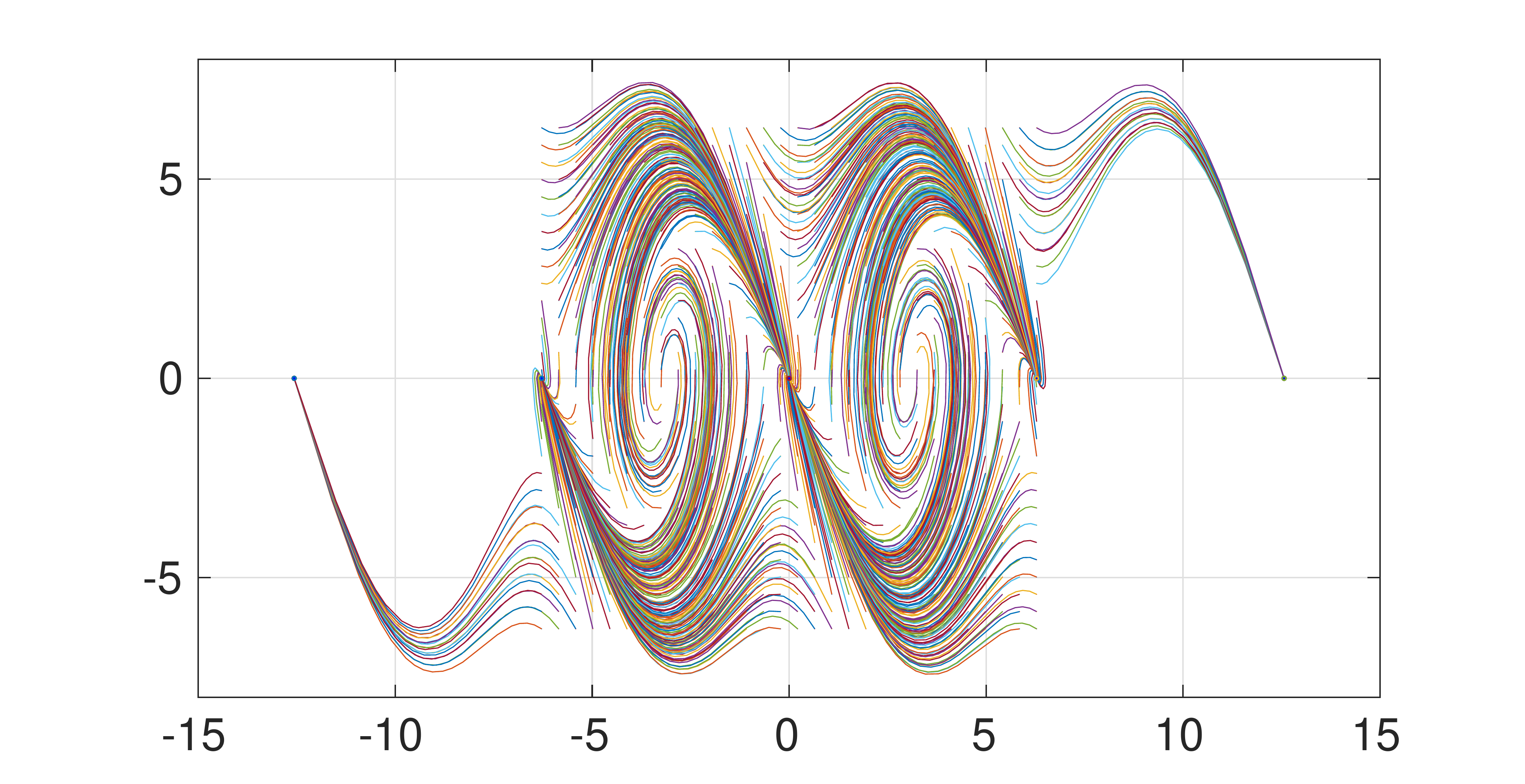}
			\end{minipage} \\
			\multicolumn{3}{c}{\smaller (a) Without control saturation $\controlset = \Real{}$}
		\end{tabular}
  \begin{tabular}{ccc}
            \hspace{-5mm}	
            \begin{minipage}{0.42\textwidth}
				\centering
				\includegraphics[width=\textwidth]{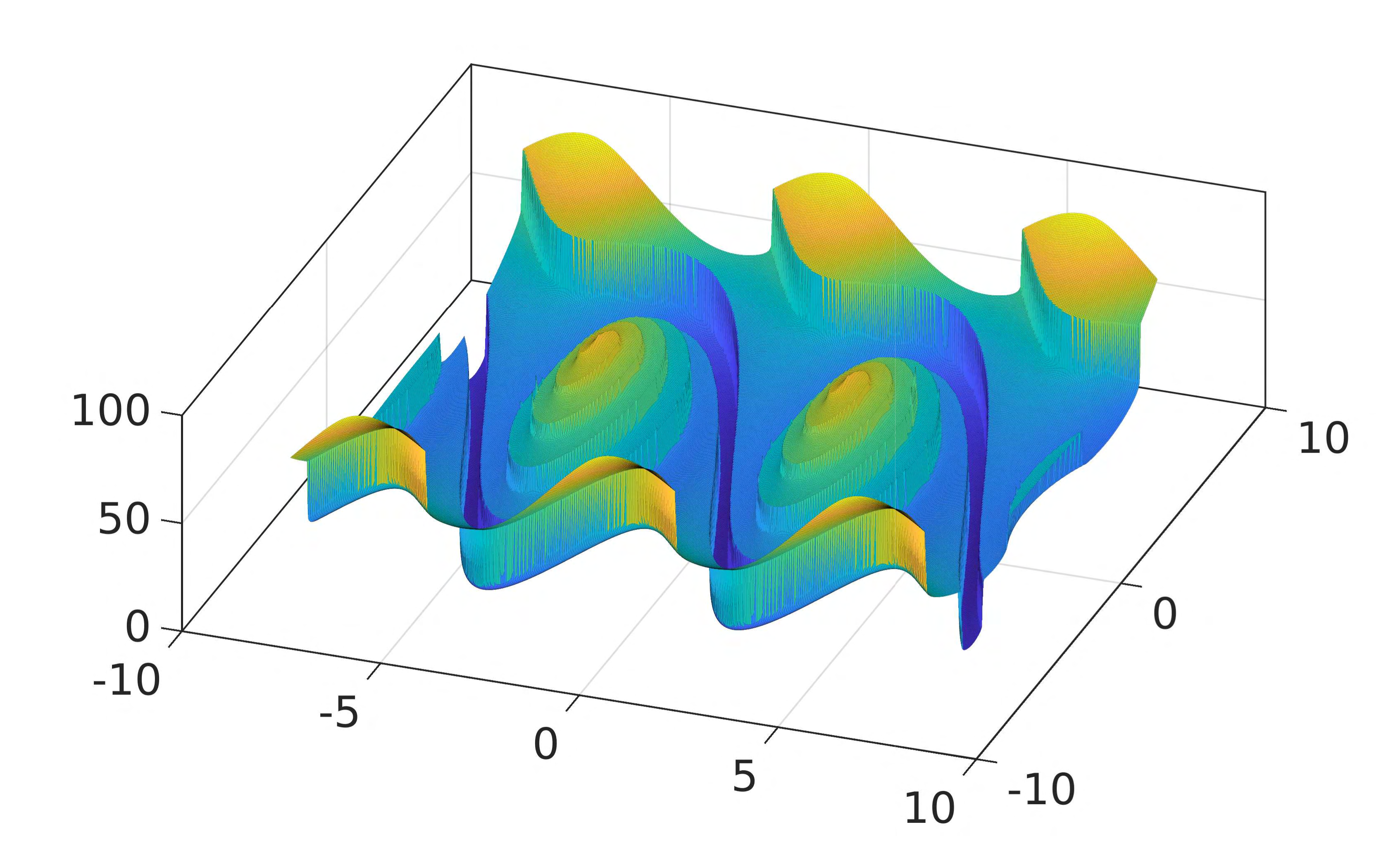}
			\end{minipage}
			&\hspace{-6mm}
			\begin{minipage}{0.28\textwidth}
				\centering
				\includegraphics[width=\textwidth]{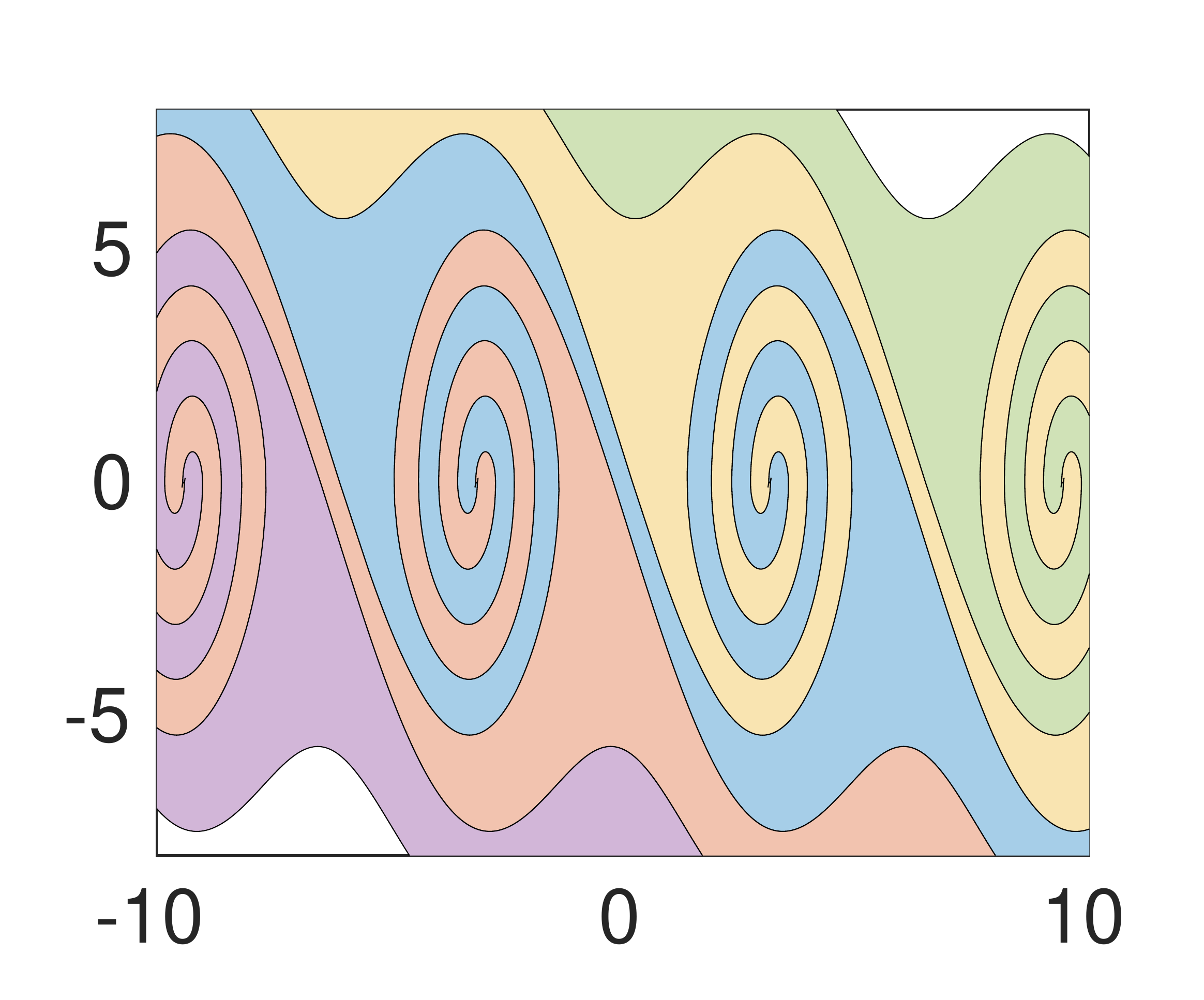}
			\end{minipage}
            &
            \hspace{-8mm}
			\begin{minipage}{0.37\textwidth}
				\centering
				\includegraphics[width=\textwidth]{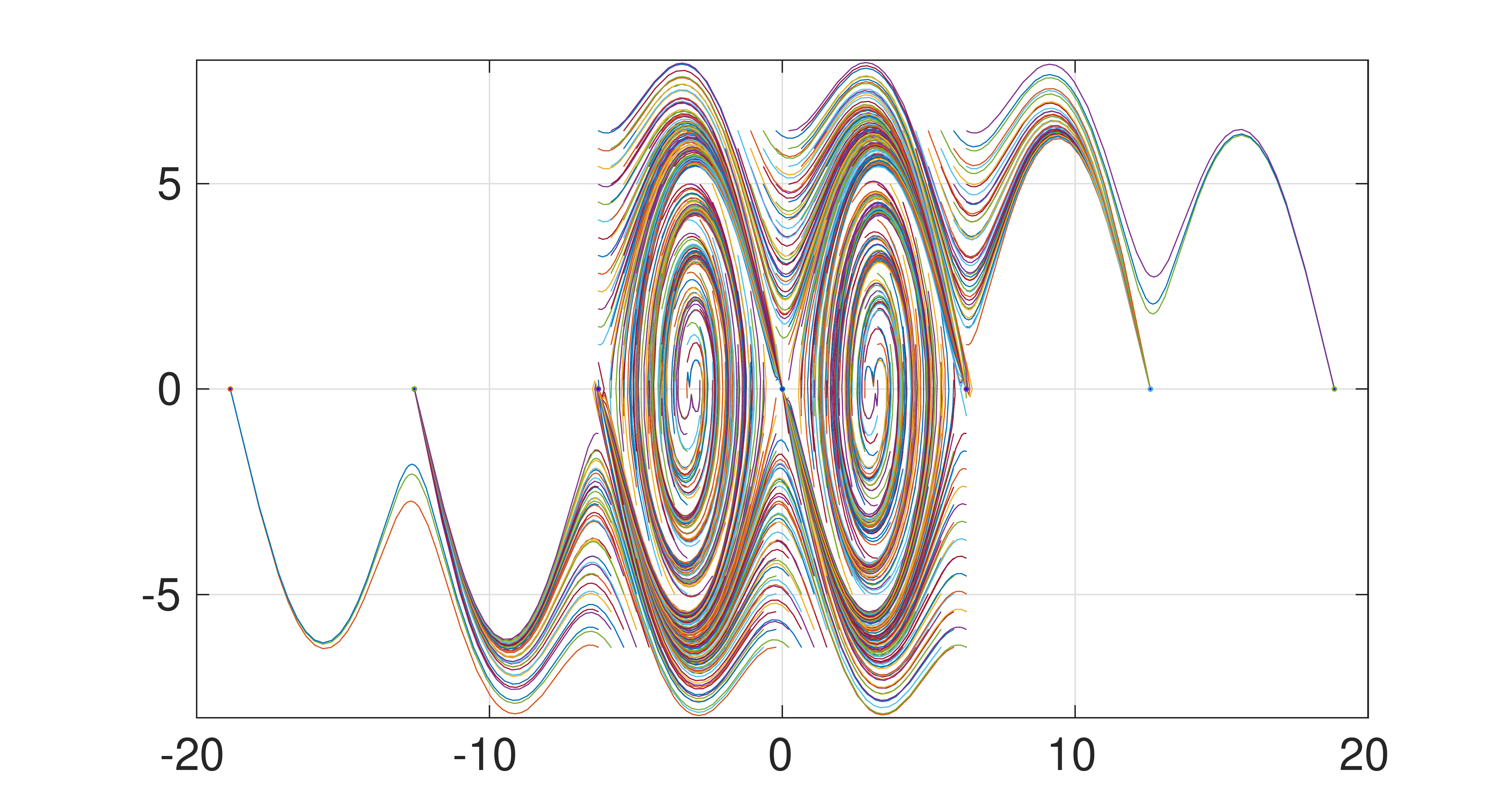}
			\end{minipage} \\
			\multicolumn{3}{c}{\smaller (b) With control saturation $\controlset = [- u_\max, u_\max], u_\max=2$}
		\end{tabular}
	\end{center}
	\vspace{-6mm}
	\caption{Optimal value function and controller. Left: optimal value function shown in 3D plots. Middle: nonsmooth curves computed from Algorithm~\ref{alg:nonsmooth-line}. Right: global stabilizing trajectories starting from $30 \times 30$ initial states. Better viewed when zoomed in. \label{fig:pendulum-value}} 
	\vspace{-1mm}
\end{figure}

{\bf Optimal Controller}.
The right column of Fig.~\ref{fig:pendulum-value} plots state trajectories using the optimal controller induced by the optimal value function via~\eqref{eq:solution-u}, starting from a dense grid of $30 \times 30$ initial states. Observe that the optimal controller swings up and stabilizes the pendulum in all cases. We then investigate if the optimal controller outperforms other algorithms. We implement four baselines: (i) energy pumping plus local LQR, (ii) open-loop trajectory optimization using direct collocation with $80$ variable timesteps, (iii) model predictive control (MPC) with $5$ seconds prediction horizon, at $50$ Hz and $100$ Hz using~\citep{fiedler23cep-do-mpc}, and (iv) proximal policy optimization (PPO)~\citep{schulman17arxiv-ppo,raffin21jmlr-stablebaseline}. Comparison with the first three baselines using the same set of parameters as before are shown in Fig.~\ref{fig:compare} right column. Observe that the optimal controller achieves lower costs than energy pumping and trajectory optimization, and almost the same cost as MPC. PPO fails in the original set of parameters but succeeds with $q_1 =1$, $q_2 = 0.1$, $r=0.01$. We therefore rerun our numerical procedure to compare our controller with PPO, shown in Fig.~\ref{fig:compare} middle column. Similarly, the optimal controller outperforms PPO in terms of lower cost.

\begin{figure}[t]
	\begin{center}
		\vspace{-6mm}
		\begin{tabular}{ccc}
            \hspace{-5mm}	
            \begin{minipage}{0.42\textwidth}
				\centering
				\includegraphics[width=\textwidth]{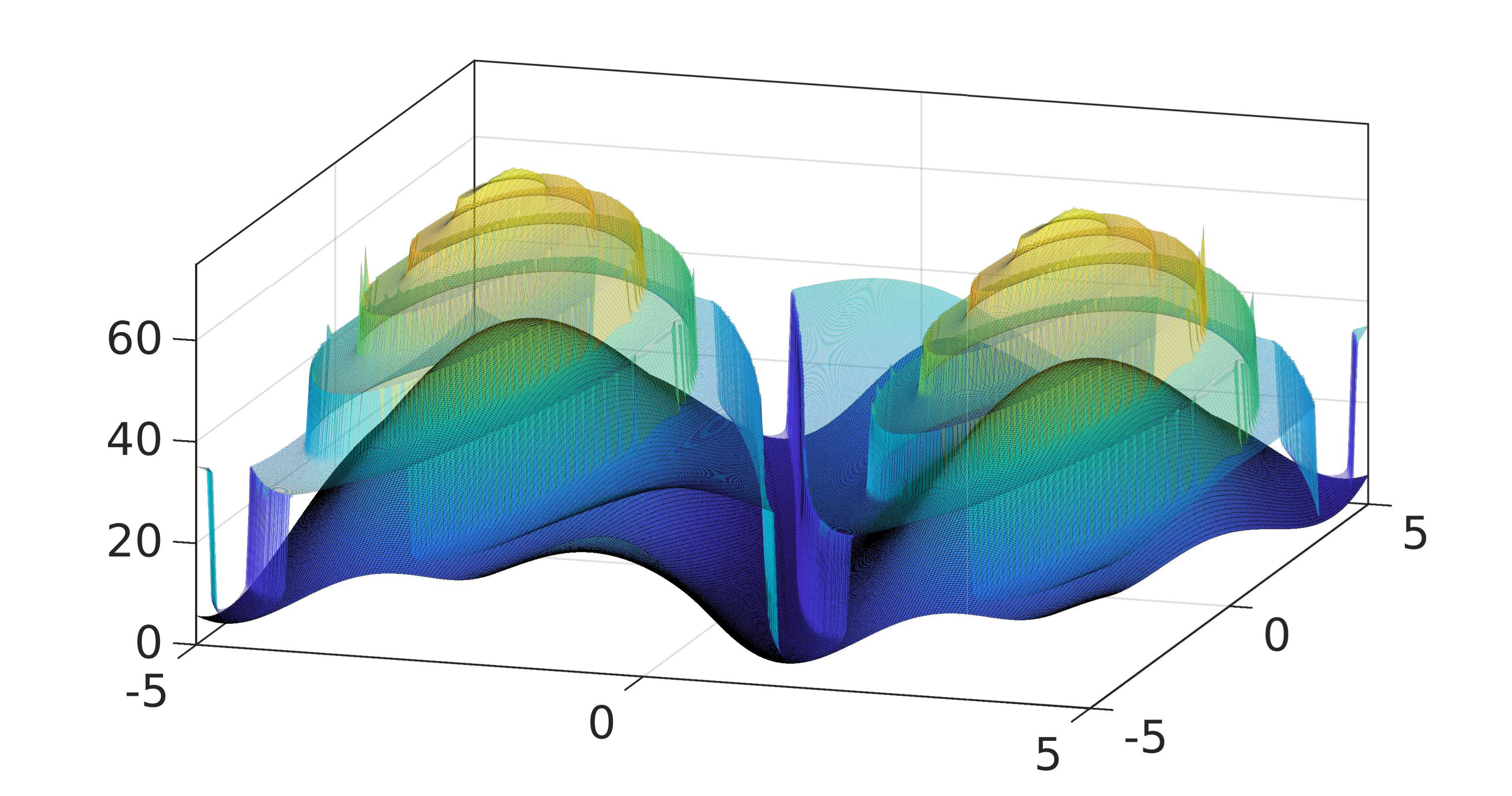}
			\end{minipage}
			&\hspace{-6mm}
			\begin{minipage}{0.29\textwidth}
				\centering
				\includegraphics[width=\textwidth]{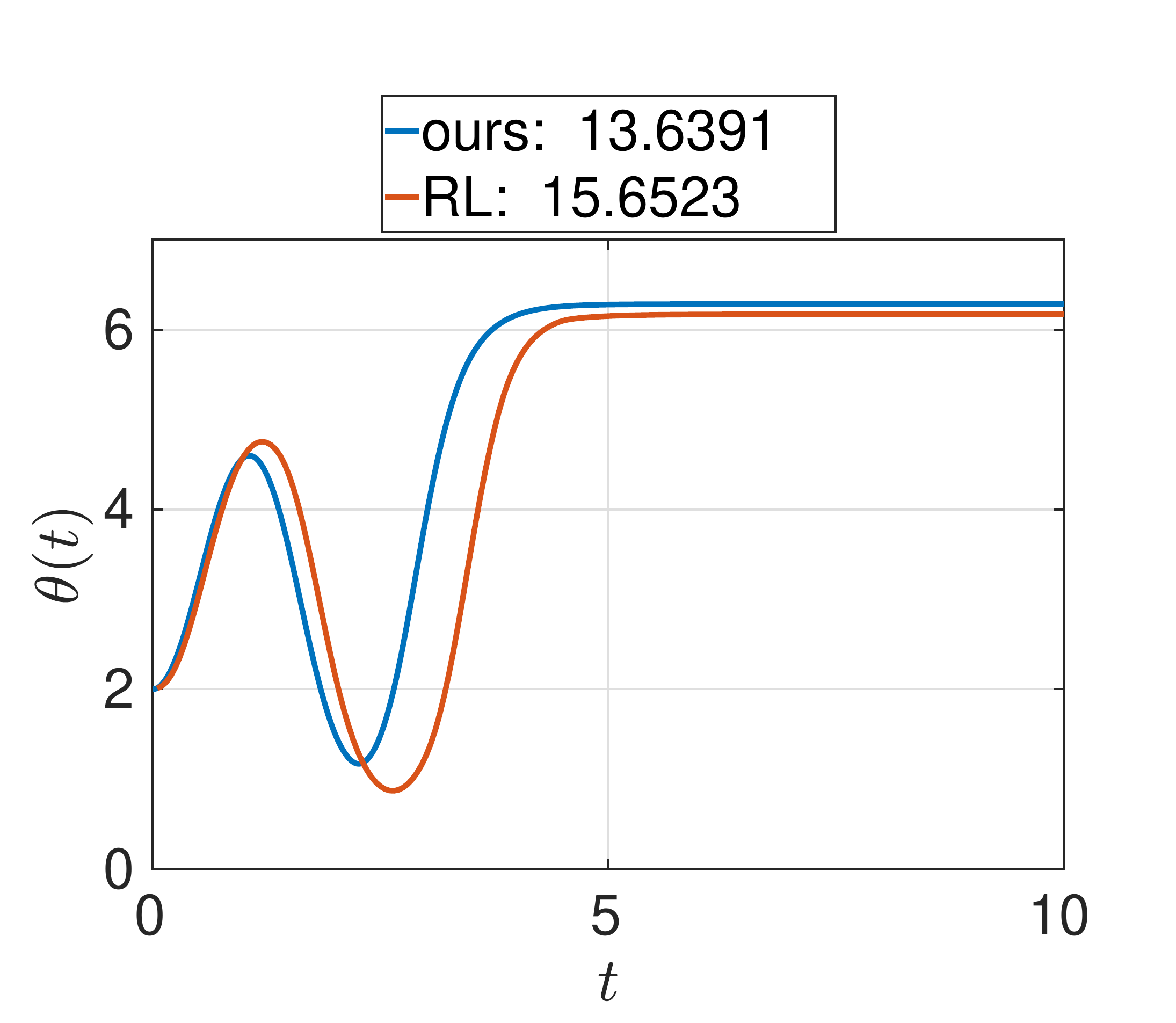}
			\end{minipage}
            &
            \hspace{-8mm}
			\begin{minipage}{0.33\textwidth}
				\centering
				\includegraphics[width=\textwidth]{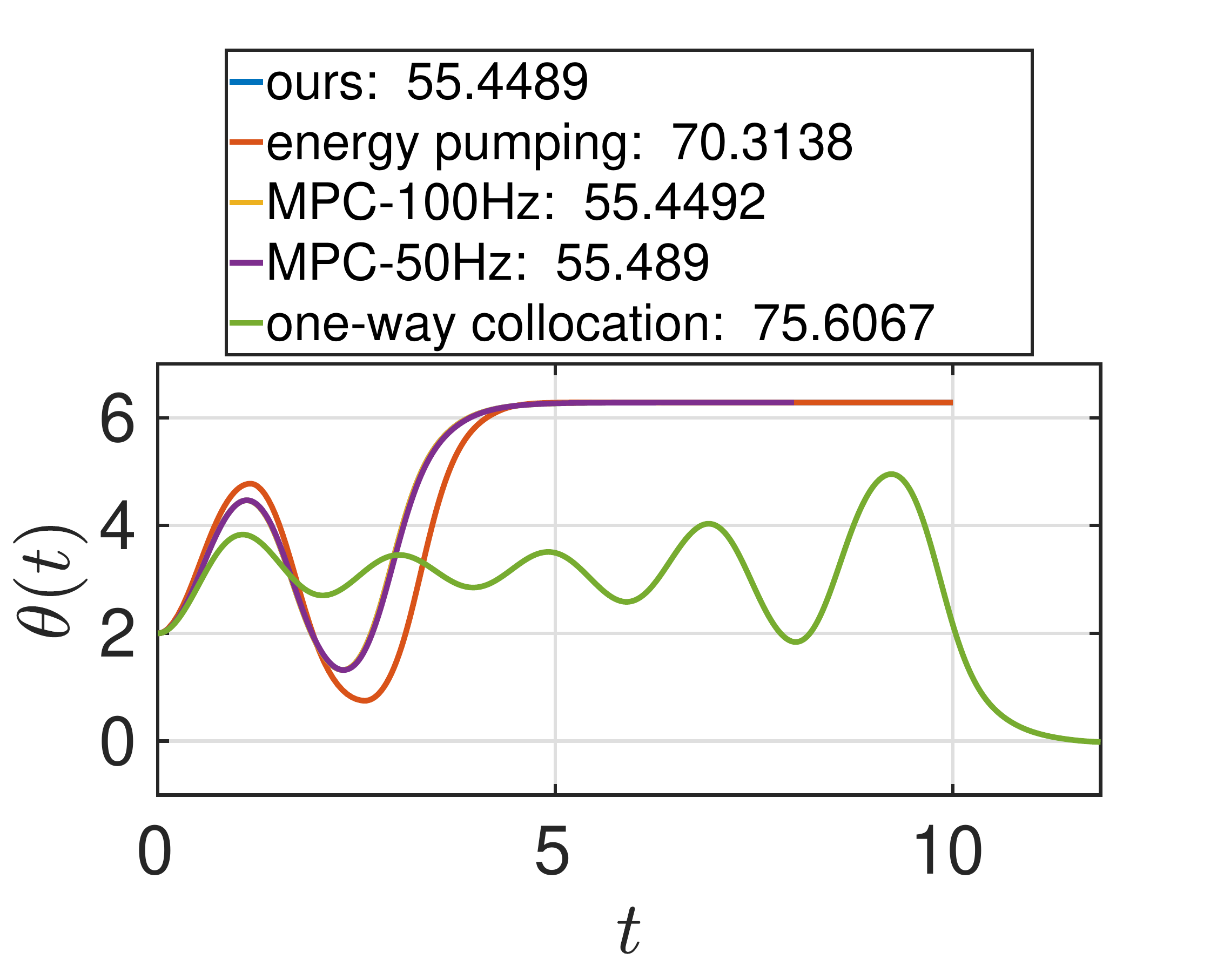}
			\end{minipage}
		\end{tabular}
	\end{center}
	\vspace{-6mm}
	\caption{Comparison of the optimal value function and controller with baselines. Left: the optimal value function sits above a polynomial lower bound. Middle: optimal controller achieves lower cost than a controller trained from PPO. Right: optimal controller achieves lower cost than energy pumping, and almost the same cost as MPC. \label{fig:compare}} 
	\vspace{-2mm}
\end{figure}



\vspace{-4mm}
\section{Neural Approximation}
\label{sec:neural-approximation}

The power of optimality comes at a price: there are $78,797$ raw PMP trajectories and $371,028,742$ value samples in the optimal value function of Fig.~\ref{fig:pendulum-value}(b). We investigate using a neural network $\Jnn(x)$ to distill knowledge from the PMP data. We use a neural network with 2 hidden layers each with $200$ neurons. The input to $\Jnn$ is $(\sin\theta, \cos\theta, \thetadot)$. We consider the case with control saturation.

{\bf Supervised Training}. We supervise $\Jnn(x)$ using data samples from $J^*(x)$ with the loss
\bea\label{eq:nn-loss}
\ell_{\supervise} = \lambda_{\lqr} \ell_{\lqr} + \lambda_{\Jvalue} \ell_{\Jvalue} + \lambda_{\hjb} \ell_{\hjb} + \lambda_{\smooth} \ell_{\smooth},
\eea
where $\ell_{\lqr}$ uses the local LQR value function $J_\infty(x)$ to supervise $\Jnn(x)$ around $\xur$; $\ell_\Jvalue$ uses random samples from $J^*(x)$ in Fig.~\ref{fig:pendulum-value} to supervise $\Jnn(x)$; $\ell_\hjb$ penalizes violation of the HJB residual~\eqref{eq:HJB}; and $\ell_\smooth$ encourages $\Jnn(x)$ to be smooth (more details in Appendix~\ref{app:sec:nn-loss-functions}). Fig.~\ref{fig:neural}(a) plots trained $\Jnn(x)$ and the induced controllers with decreasing samples used in $\ell_{\Jvalue}$. We see even with just $50$ value samples, the controller globally swings up and stabilizes the pendulum.

{\bf Weakly Supervised Training}.
The loss $\ell_\Jvalue$ requires $J^*(x)$ that is expensive to compute due to Algorithm~\ref{alg:nonsmooth-line}. We replace $\ell_\Jvalue$ with a loss that only requires \emph{raw} PMP trajectories
$$
\ell_\pmp = \frac{1}{N_{\pmp}} \sum_{i=1}^{N_\pmp} \mathrm{LeakyReLU}(\Jnn(x_i) - \pmp(x_i)),
$$
where $\pmp(x_i)$ indicates the value of $x_i$ along a given PMP trajectory. Choosing $N_{\pmp} = 100000$, we obtain $\Jnn(x)$ and its induced controller that globally stabilizes the pendulum in Fig.~\ref{fig:neural}(b). In Appendix~\ref{app:sec:cart-pole} we show the weakly supervised method generalizes to the $3$-dimensional cart-pole.

\begin{figure}[h]
	\begin{center}
		\begin{tabular}{cccc}
            \hspace{-5mm}	
            \begin{minipage}{0.27\textwidth}
				\centering
				\includegraphics[width=\textwidth]{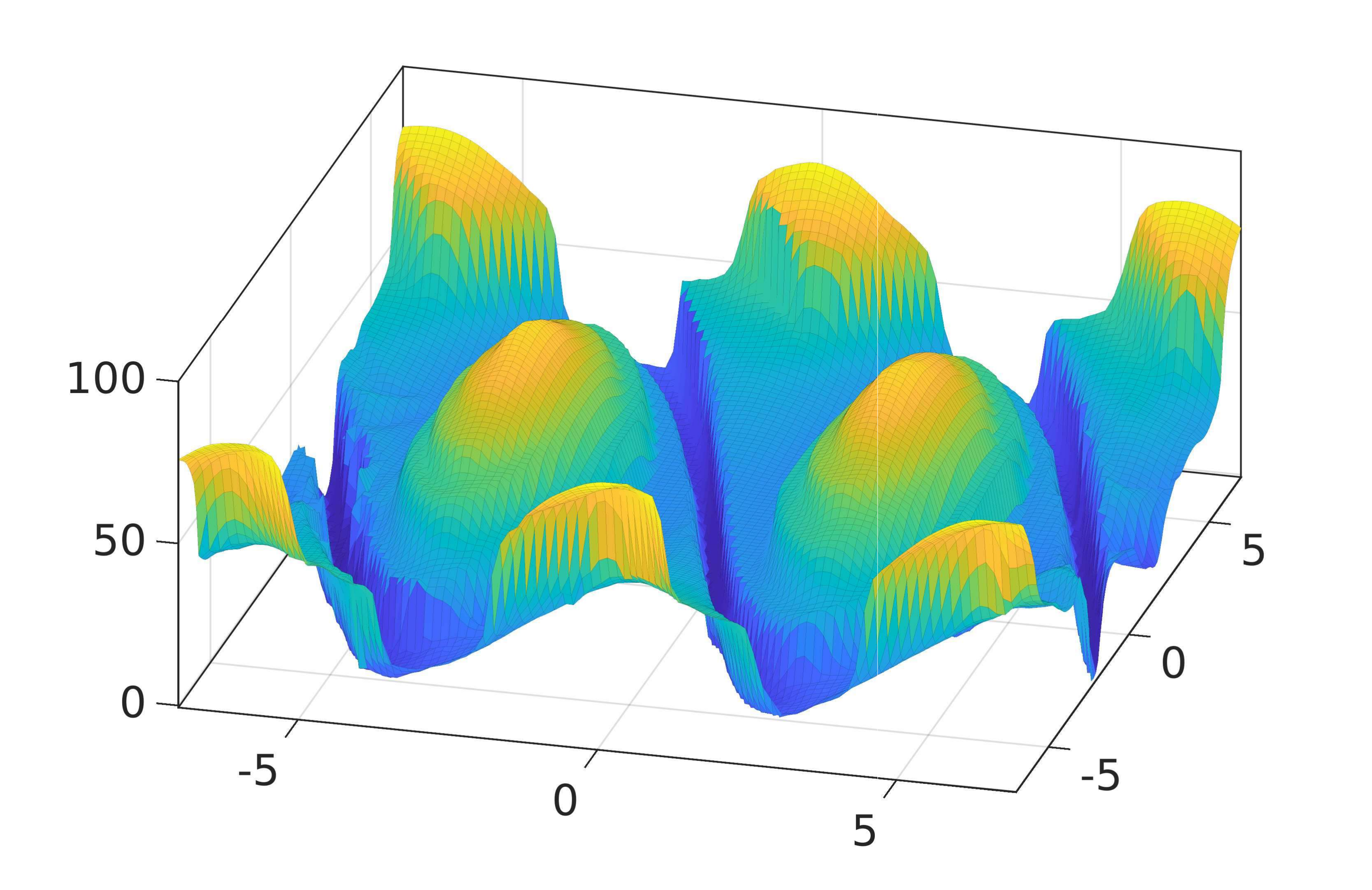}
			\end{minipage}
			&\hspace{-6mm}
			\begin{minipage}{0.27\textwidth}
				\centering
				\includegraphics[width=\textwidth]{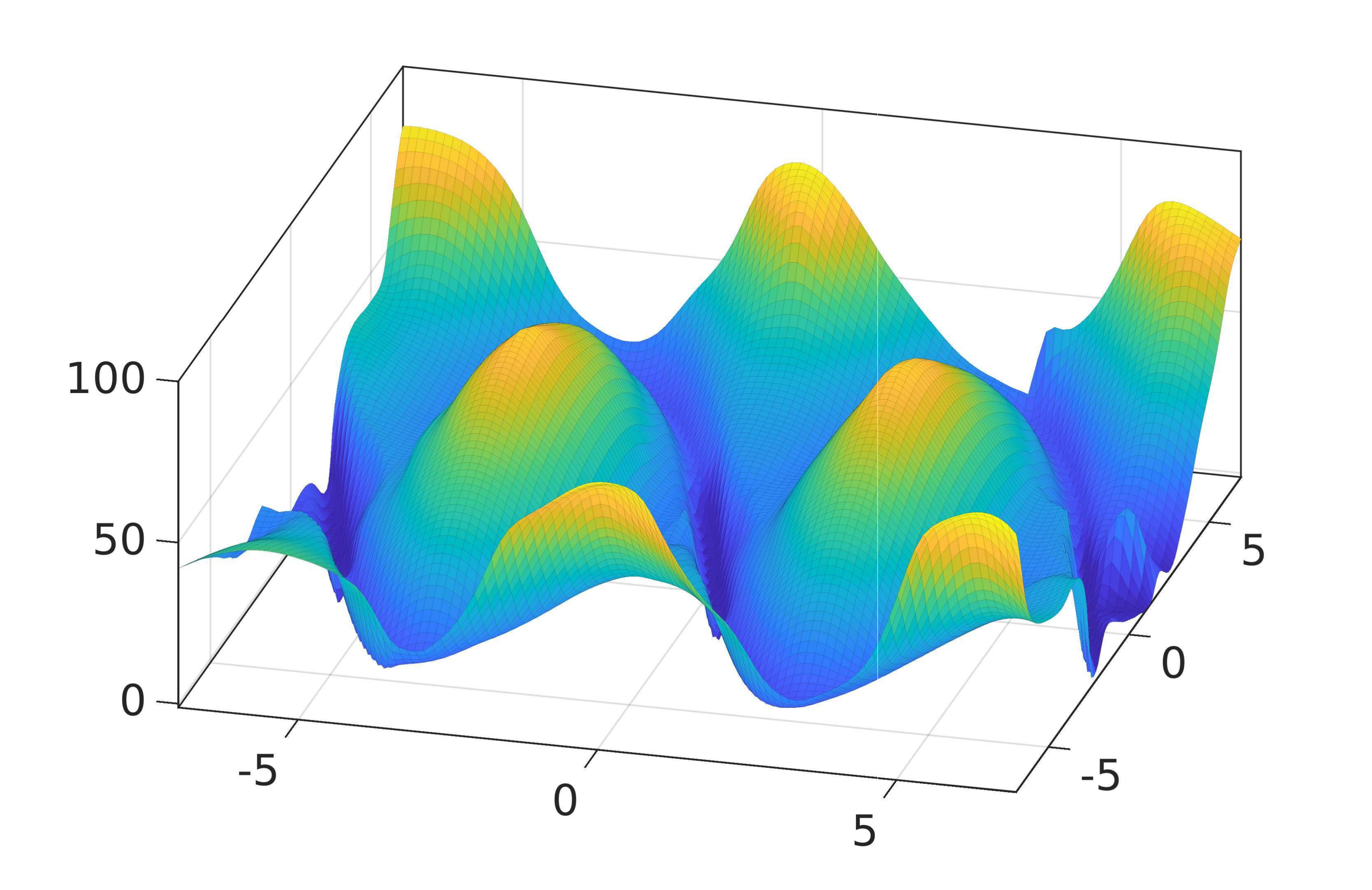}
			\end{minipage}
            &
            \hspace{-8mm}
			\begin{minipage}{0.27\textwidth}
				\centering
				\includegraphics[width=\textwidth]{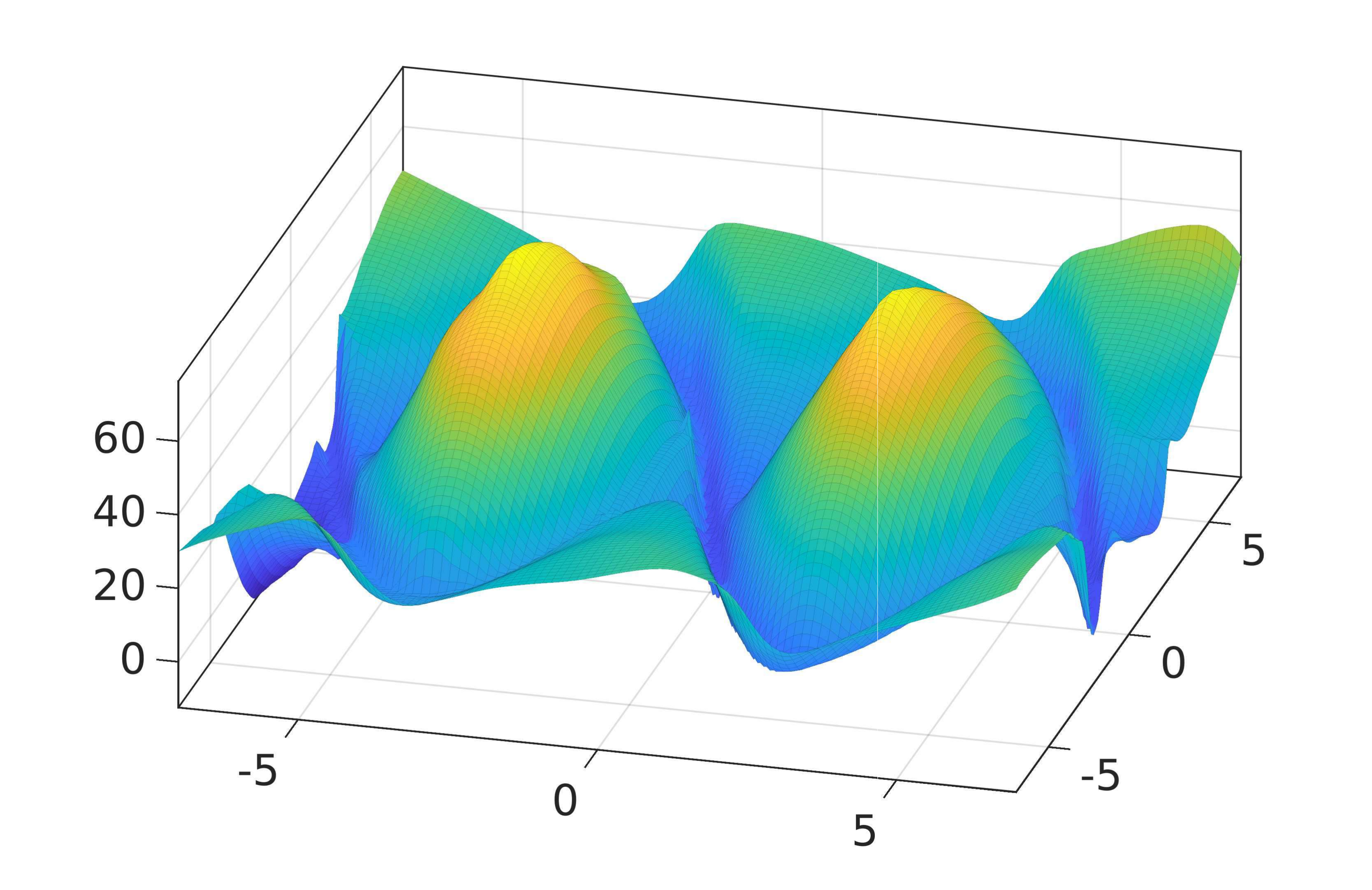}
			\end{minipage}
            &
            \hspace{-8mm}
			\begin{minipage}{0.27\textwidth}
				\centering
				\includegraphics[width=\textwidth]{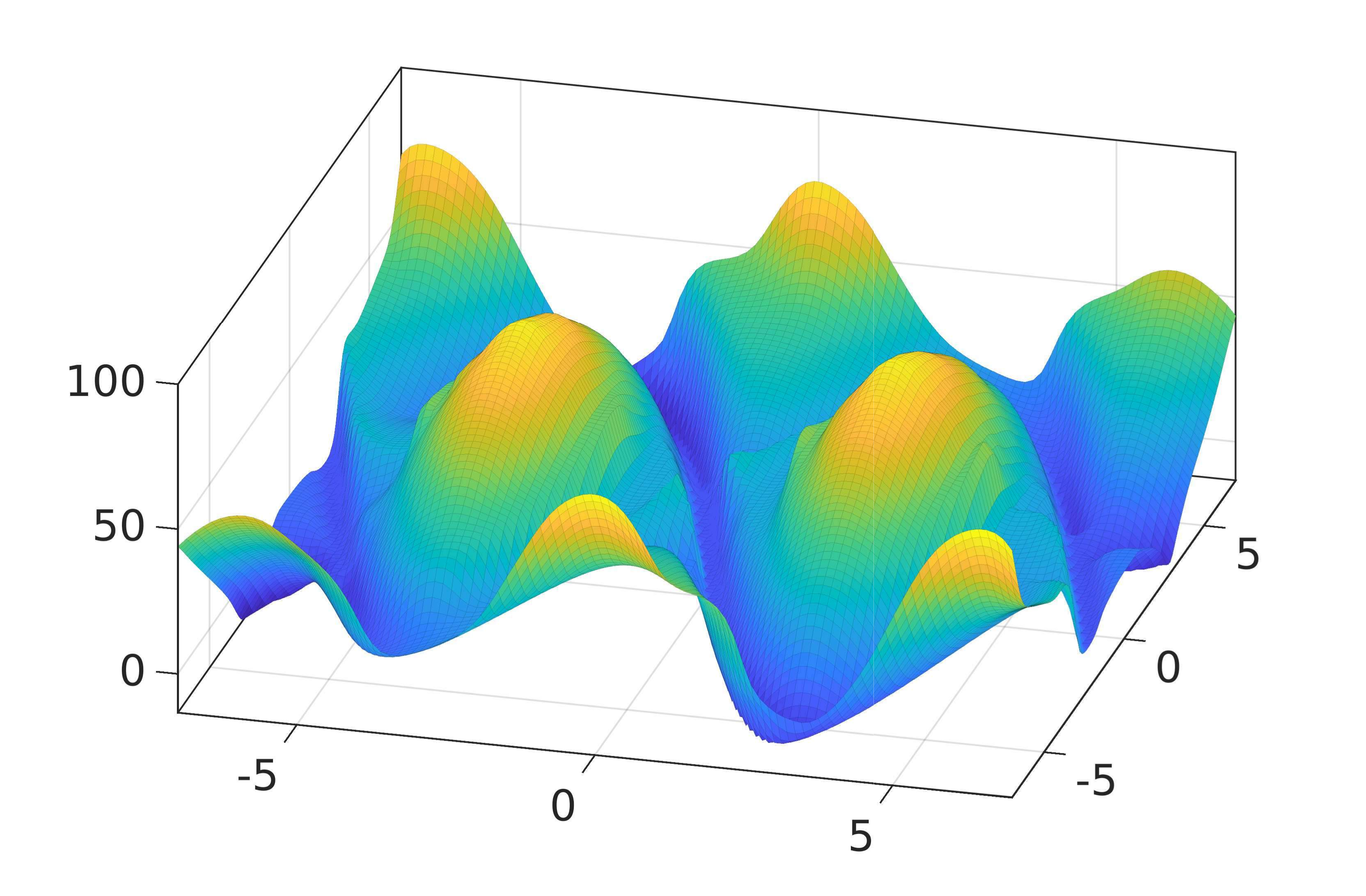}
			\end{minipage}
		\end{tabular}
\begin{tabular}{cccc}
            \hspace{-5mm}	
            \begin{minipage}{0.27\textwidth}
				\centering
				\includegraphics[width=\textwidth]{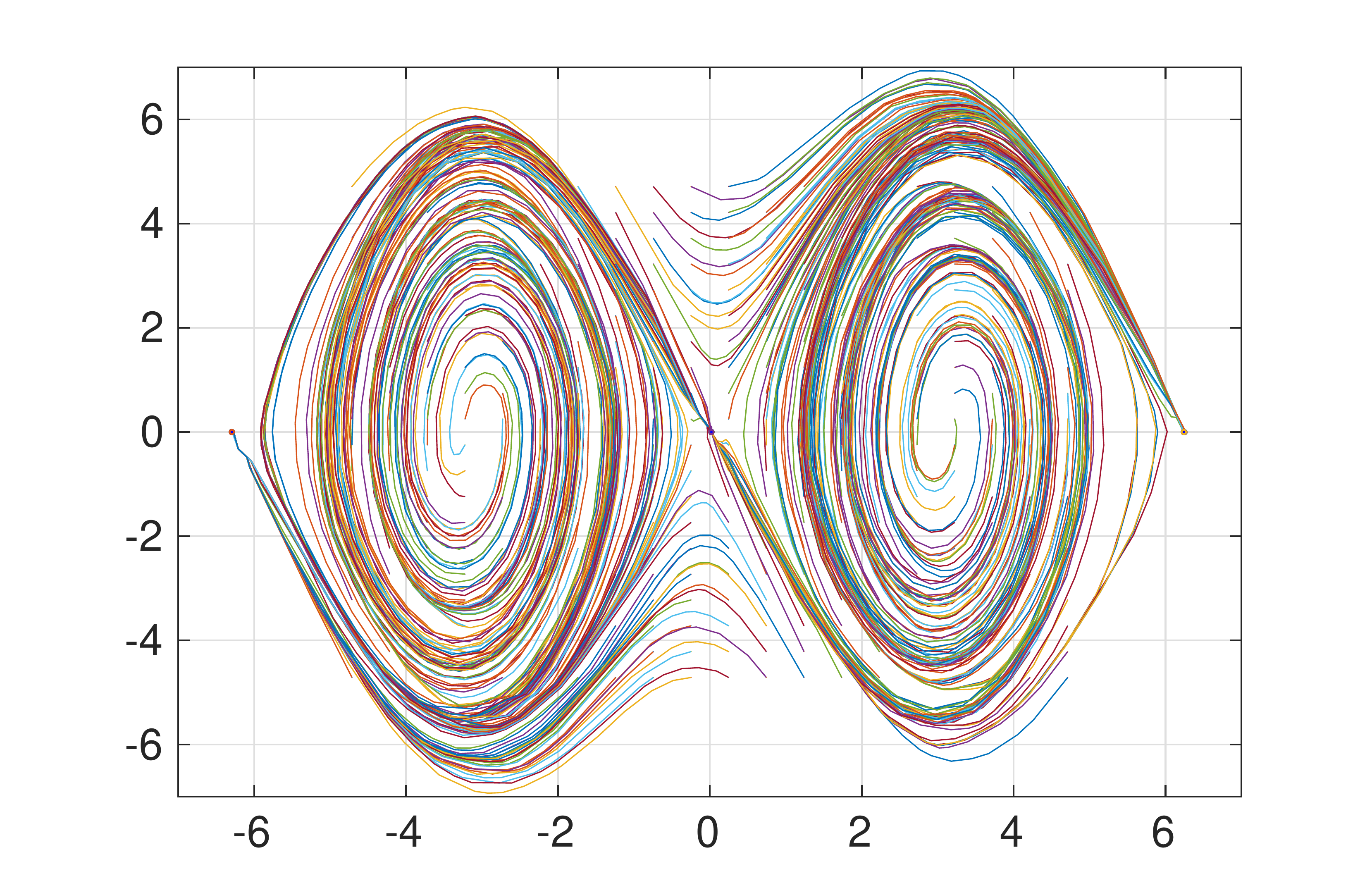}
			\end{minipage}
			&\hspace{-6mm}
			\begin{minipage}{0.27\textwidth}
				\centering
				\includegraphics[width=\textwidth]{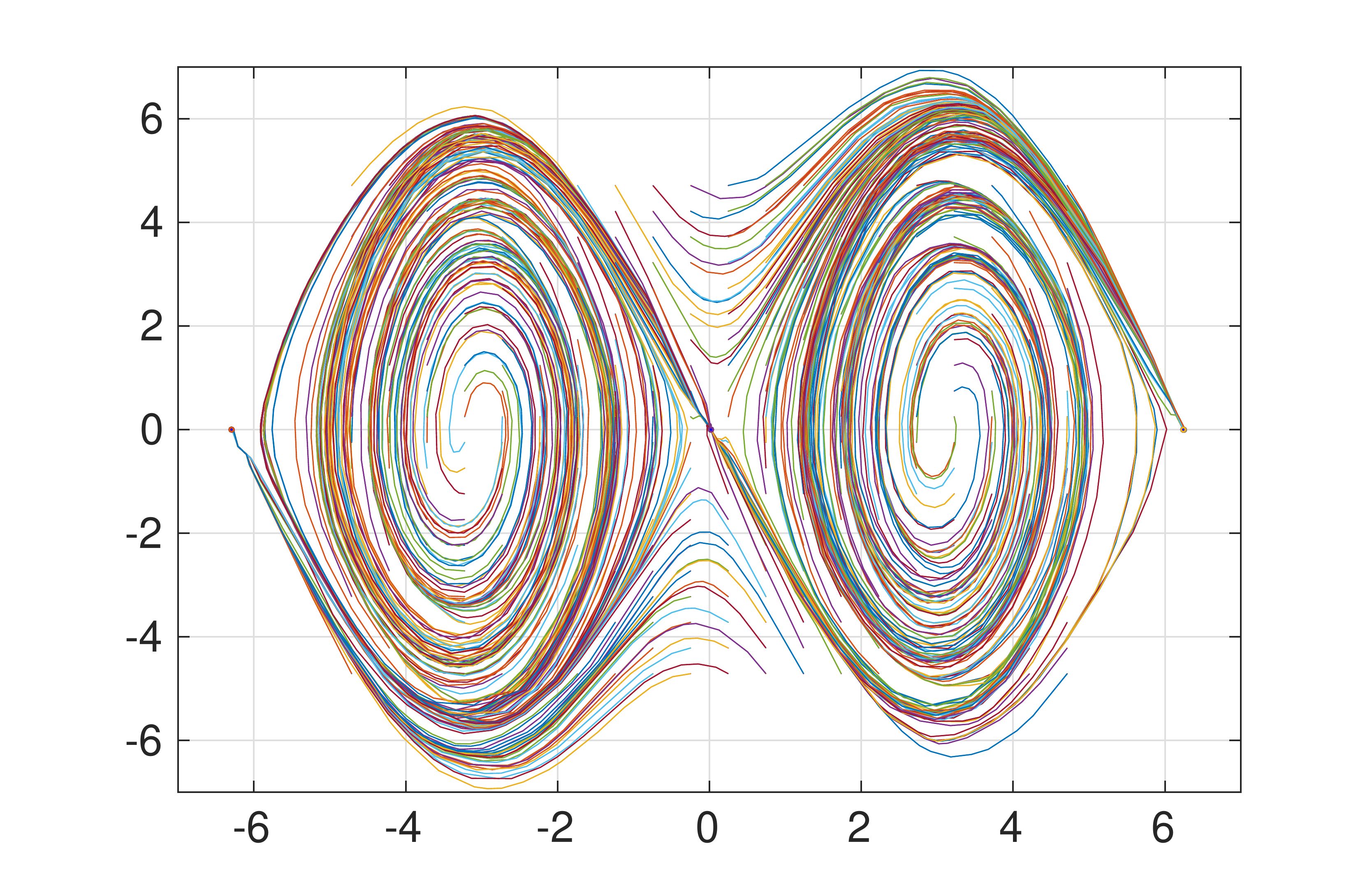}
			\end{minipage}
            &
            \hspace{-8mm}
			\begin{minipage}{0.27\textwidth}
				\centering
				\includegraphics[width=\textwidth]{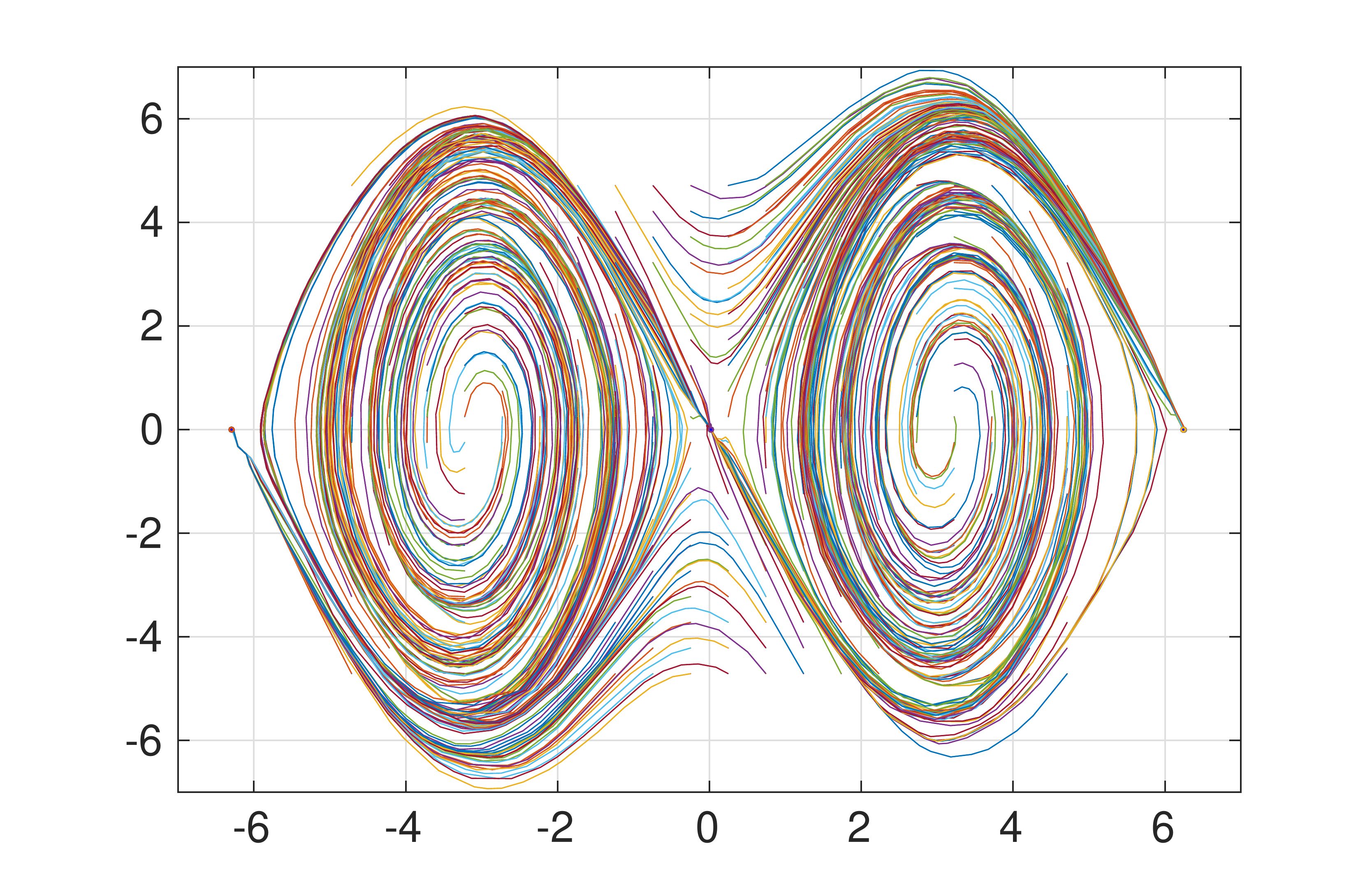}
			\end{minipage}
            &
            \hspace{-8mm}
			\begin{minipage}{0.27\textwidth}
				\centering
				\includegraphics[width=\textwidth]{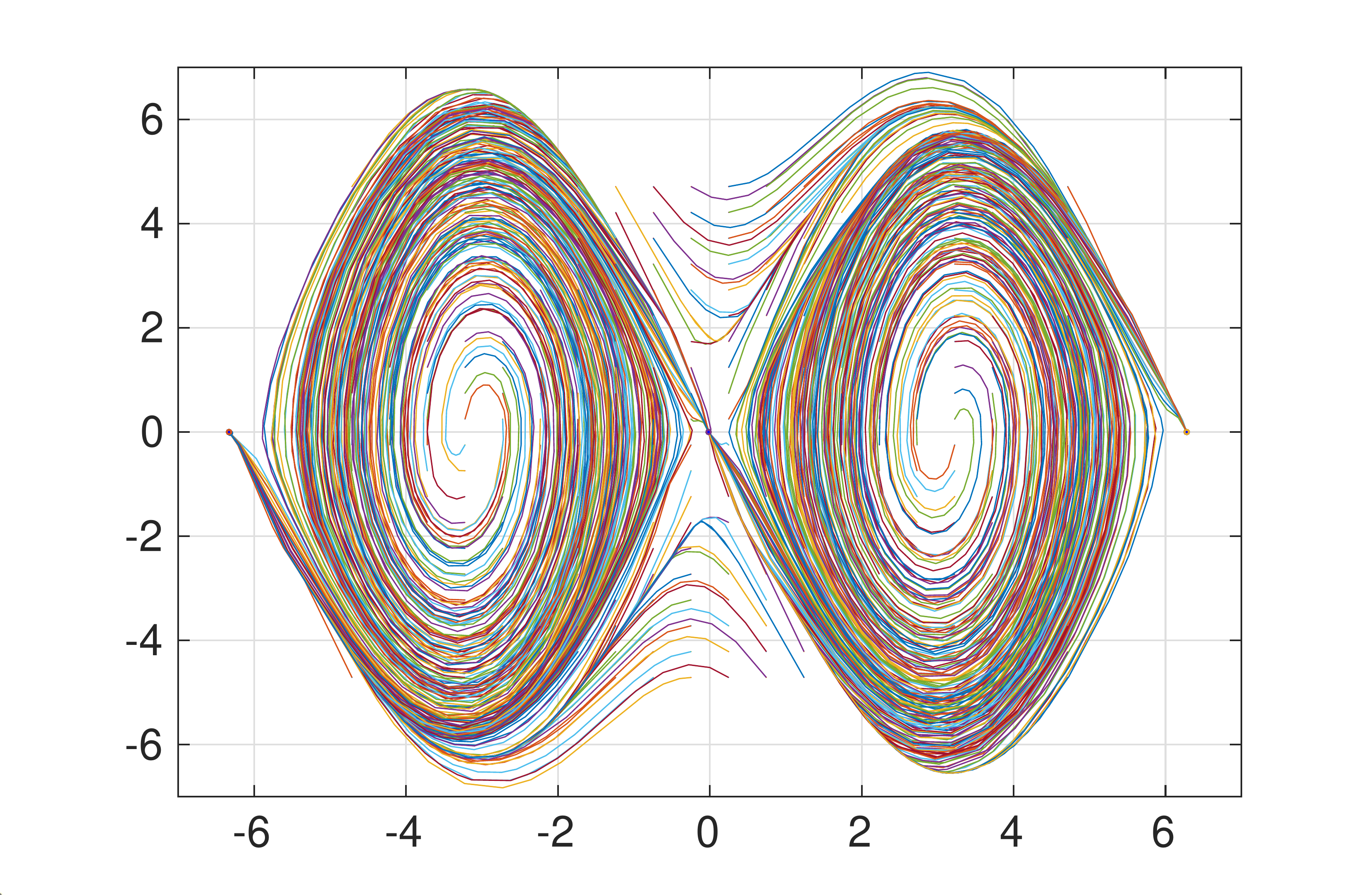}
			\end{minipage} \\
			\multicolumn{3}{c}{\smaller (a) Supervised training with, left to right, $1000$, $100$, and $50$ samples} & \hspace{-8mm} {\smaller (b) Weak supervision}
		\end{tabular}
	\end{center}
	\vspace{-6mm}
	\caption{Neural approximations of the optimal value function. (a) Supervised training with decreasing data samples. (b) Weakly supervised training with raw PMP trajectories.\label{fig:neural}} 
	\vspace{2mm}
\end{figure}



\vspace{-4mm}
\section{Conclusion}
We showed the optimal value function of infinite-horizon undiscounted pendulum swing-up is nonsmooth. Motivated by this, theoretically, we provide two results that certify the optimality and suboptimality of candidate value functions; algorithmically, we develop a numerical procedure based on backward solving PMP with local LQR terminal conditions to compute the true optimal value function up to minor numerical inaccuracies. The optimal value function outperforms other baseline algorithms and verified optimality. We demonstrate it is possible to learn simple and effective neural approximations of the optimal value function via either strong or weak supervision.

\acks{
  We thank Michael Posa, Jean-Bernard Lasserre, Jiarui Li, Yukai Tang, and Shucheng Kang for discussions about the optimal pendulum swing-up problem; Didier Henrion, Zexiang Liu, and Necmiye Ozay for pointing us to several related works; Lujie Yang and Alexandre Amice for help with computing polynomial lower bounds using SOS programming.
}

\appendix
\setcounter{equation}{0}
\setcounter{table}{0}
\setcounter{figure}{0}
\setcounter{theorem}{0}
\renewcommand{\theequation}{A\arabic{equation}}
\renewcommand{\theproposition}{A\arabic{proposition}}
\renewcommand{\thetheorem}{A\arabic{theorem}}
\renewcommand{\theassumption}{A\arabic{assumption}}
\renewcommand{\thefigure}{A\arabic{figure}}
\renewcommand{\thetable}{A\arabic{table}}


\section{Proof of Theorem~\ref{the:non-smooth}}
\label{sec:app:proof-nonsmoothnes}

\subsection{Existence and Uniqueness Theorem of ODE}
\begin{lemma}[Uniqueness Theorem of ODE]\label{the:uniqueness theory}
    Consider the initial value problem
    \beq
    \begin{cases}
    \displaystyle\frac{\partial y}{\partial x} = g(x,y)\\
    y(x_0)
 = y_0
\end{cases}.
    \eeq
    Assume $g(x,y)$ is Lipschitz continuous in $R : = \{(x,y) \mid \Vert y-y_0\Vert \leq b,\Vert x-x_0\Vert\leq a \}$, then there is only one solution $y(x)$ when $x \in I = [x_0-h,x_0+h]$, where
    $$M = \max_{(x,y) \in R}\Vert g(x,y)\Vert, \quad h = \min\{a,\frac{b}{M}\}$$
\end{lemma}

\subsection{Proof of Theorem ~\ref{the:non-smooth}}

\begin{proof}
    The proof is organized as follows: Firstly we show that if the value function satisfies the HJB equation at $x(0) = \xbr$ then a unique optimal controller $u(0)$ can be solved. Secondly, we show for $\xbr$ there are two equivalent controls $u(0)$ and $-u(0)$, and furthermore, $u(0) = 0$ is impossible, thus leading to a contradiction.

    {\bf Contradiction between strong convexity and symmetry}. Assuming the optimal value function $J$ is $C^1$ everywhere near $\xbr$, then due to the Bellman Principle using dynamics programming, the value function $J$ satisfies HJB equation \eqref{eq:HJB} everywhere near a neighborhood $\mathcal{B}$ of $\xbr$. Now in \eqref{eq:HJB} if we fix $x$, and considering $f$ is affine in $u$, the objective of the optimization in~\eqref{eq:HJB}
    $$F(u) = q_1 \sin^2 \theta + q_1 (\cos \theta -1 )^2 + q_2 \thetadot^2+  r u^2 +  \frac{\partial J}{\partial x}\tran f(x,u)$$
    is a strongly convex function, and $\controlset$ is a convex set, so the HJB equation has a unique solution $u^*$.
    Actually we can solve out $u$ explicitly as in~\eqref{eq:solution-u}.

    But if an optimal trajectory $(\theta(t),\thetadot(t))$ and an optimal control $u(t)$ exist, then $(-\theta(t),-\thetadot(t))$ and $-u(t)$ are also optimal due to symmetry of the pendulum problem shown in Fig.~\ref{fig:pendulum}. Thus $u(0)$ and $-u(0)$ are both optimal controls at $\xbr$. Then we only need to show $u(0)$ is not zero. We plug in $u = 0$ into $F(u)$ when $\theta = \pi$ and $\dot\theta = 0$, then we get $F(0) = 4q_1 \neq 0$.
    This conflicts with the HJB equation.



    But now there are two nonzero optimal controllers: $u(0)$ and $-u(0)$, which is contradictory to the HJB equation. Therefore, at $\xbr$ the optimal value function must be nonsmooth.
\end{proof}

\section{Proof of Theorem~\ref{the:optimality}}
\label{sec:app:proof-optimality}

\begin{proof}
The proof is organized as follows: Firstly we prove two observations that $x(t)$ is piece-wise $C^1$ and all the intersection points between any admissible trajectory and the nonsmooth curve $\nonsmoothregion$ is a closed set. Secondly, we prove the intersection points parameterized by $t$ will be either an interval or a single point. Finally, we apply the HJB equation to every interval and show the piece-wise $C^1$ function is a lower bound to the optimal value function, whose optimality then follows from the attainability condition (vii).

First we observe $x(t)$ is piece-wise $C^1$. 
The dynamics is
$$
\dot x = f(x,u).
$$
Because $u(t)$ is piece-wise continuous, $f$ is continuous, so on every interval that $u(t)$ is continuous, $\xdot$ is continuous and $x(t)$ is $C^1$.


Without loss of generality, assume $x(t)$ is always $C^1$ on $t$ (we will discuss piece-wise $C^1$ later) and the intersection of $x(t)$ with the nonsmooth curve $\nonsmoothregion $ is 
$$
\{x(t)|t \in I\}.
$$

Next we prove $I$ is a closed set on $\Real{1}$, \ie for every converging sequence the limit is in set $I$. For a converging sequence we only need to consider a uniform neighborhood $\local$ of the limit. From $\nonsmoothregion  : \{G(x) = 0|x\in \local\}$,  we know $\{x(t)|t \in I, x(t)\in\local\} = \{x(t)|G(x(t))=0, x(t)\in \local\}$. Since $G$ and $x$ are both $C^1$, and $C^1$ function's zero point set is closed, we know $I$ is closed.


We define a ``switching point'' (S-point) in $x(t)$: a point $x(t_0)$is called an S-point if
$$
\forall \epsilon>0,\exists t_1,t_2\in [t_0-\epsilon,t_0+\epsilon],s.t.~t_1\in I,t_2\notin I
$$
It is obvious that all S-points~$x(t_s)$ satisfy $t_s\in I$: because $I$ is a closed set, and we choose $\epsilon_n\rightarrow 0$, we get a sequence converging to  $x(t_s)$.

S-point actually contains both pathological 'bad'\footnote{Here 'bad' means it cross $\nonsmoothregion$ in a pathological way, like $sin(\frac{1}{x})$. And 'good' means what we really want: the endpoints of intervals or single point.} points and 'good' points. There are two kinds of good points: (i) the interval endpoints and (ii) a single point. We want to prove the S-points can be sorted on $\Real{1}$, so there will only be good points.

\begin{lemma}[Finite number of S-points]\label{the:finite-S}
    For $t\in[-R,R]$, $R \in \bbN$, there are only a finite number of S-points.
\end{lemma}

\begin{proof}
First, if there are infinitely many S-points we will find a special one and prove it cannot be a good point. Then, we use proof by contradiction from condition (vi).

Assume there are infinitely many S-points, since $[-R,R]$ is a bounded set, so these points have at least one limit point, \ie $x(s_i)\rightarrow x(s)$, $x(s_i)$ are different S-points. $x(s)$ satisfies
\beq\label{eq:proof-2-1}\forall \epsilon,\exists s_1,s_2\in [s-\epsilon,s+\epsilon]/\{s\},s.t.~s_1\in I,s_2\notin I
\eeq

This equation has a punctured neighborhood because we can always find an S-point different from $x(s)$ in an arbitrary small neighborhood.

Note that there must be both infinite number of $x(t) \in I$ and $x(t) \notin I$ on one side of $x(s)$: we can take sequence $x(s_i)$ approaching $x(s)$ from one side, let's say $s_i\rightarrow s^+$. So for $\forall~\epsilon>0$, $\exists s_i \neq s$, $s.t.~s_i \in (s,s+\epsilon]$, then for $\delta = \frac{1}{2}\min\{s_i-s,s+\epsilon-s_i\}$, $\exists~s_1,s_2 \in [s_i-\delta,s_i+\delta]~s.t.~s_1\in I,s_2\notin I$. This means we can choose $s_i$ from one side of $s$ in \eqref{eq:proof-2-1}.

Actually punctured neighborhood is aimed to exclude the case that $x(s)$ is a second-type good point, and $s_i$ from one side is aimed to exclude first-type good point. Now $s$ is not a good point, and then we show the contradiction.

As mentioned above, $\nonsmoothregion $ can be locally written as $\nonsmoothregion  : \{G(x) = 0|x\in \local\}$, so on $x(t)$ we get $\nonsmoothregion  : \{G(x(t)) = 0|x(t)\in \local\}$, $G(x(\cdot))$ is a $C^1$ function, and we have assumed it is monotonic near $s$. Without loss of generality, it is increasing in $t \in [s-h,s+h]$, and $\forall 0<h_\epsilon\leq h$, $\exists s_1,s_2 \in (s,s+h_\epsilon) ~s.t. ~s_1\in I,s_2\notin I$. For $h_\epsilon:=h , s_{\mathrm{in}}\in I, s_{\mathrm{in}} < h$, we get $G(x(s)) = 0,~G(x(s_{\mathrm{in}})) = 0$ so $G(x(t))\equiv 0, t\in[s,s_{\mathrm{in}}]$.
But for $h_\epsilon:=s_{\mathrm{in}}-s,s_{\mathrm{noin}}\in I, s_{\mathrm{noin}} < s_{\mathrm{in}}$, so $G(x(s_{\mathrm{noin}}))\neq 0$, which induces the confliction. 
\end{proof}

So there are a finite number of S-points in $[-R,R]$. Pushing $R\rightarrow \infty$ makes sortable number of S-points in $\Real{1}$. We sort these points as $...,T_{-1},T_1,T_2,...$ and in each open interval $(T_i,T_{i+1})$, we prove that this interval is either all on $I$ or all not on $I$.

\begin{proof}
Proof by contradiction: Because every point is not an S-point, we get:
$$
\forall t \in (T_i,T_{i+1}), \exists \epsilon_t, s.t. \forall s \in [t-\epsilon_t,t+\epsilon_t], s~\mathrm{is~either~on~or~not~on~} I.
$$ 
Assume there are two points $T_{\mathrm{in}} \in I$ and $T_{\mathrm{noin}} \notin I$, for $[T_{\mathrm{in}},\frac{T_{\mathrm{in}}+T_{\mathrm{noin}}}{2}]$ and $[\frac{T_{\mathrm{in}}+T_{\mathrm{noin}}}{2}, T_{\mathrm{noin}}]$, there must be at least one interval that has points both in and not in $I$. Do this repeatedly, and it will converge to a point, this point is a S-point by definition.
\end{proof}

For $x(t)$ that is piece-wise $C^1$, we can add the nondifferentiable point in $\{T_i|i\in \bbZ\}$ and use the same result above in every interval.


Now every open interval is either all on or all not on $\nonsmoothregion $. For those who are not on $\nonsmoothregion $ it stays in one $\smoothregion_i $ (otherwise it will touch $\nonsmoothregion$), so HJB holds everywhere. For those who lie on $\nonsmoothregion $, it means that on that point, there exists unique $(i,j)$, $J_i(x) = J_j(x) = J(x)$, so $x \in \smoothregion_i $. This is because condition (ii): $\smoothregion_i  \cap \smoothregion_{i+v} = \emptyset,|v| > 1$, so the pair $(i,j)$ is unique. Now on each interval $[T_i, T_{i+1}]$ we have (a) $u(t)$ is continuous, (b) HJB equation holds everywhere.

From
$$\min_{u \in \controlset}c(x,u)+\frac{\partial J}{\partial x}\tran f(x,u) = 0.\quad \forall x \in \smoothregion_i$$
we get
$$c(x,u)+\frac{\partial J}{\partial x}\tran f(x,u) \geq 0,\quad \forall x\in \smoothregion_i ,\quad \forall u \in \controlset$$
Note that $J$ is continuous and $x(t)$ is continuous (by definition), so $J(x(t))$ is continuous. 

Thus,
$$\frac{dJ}{dt} = \frac{dJ_i}{dt} = \frac{\partial J_i}{\partial x}\tran \dot x\geq -c(x(t),u(t)),\quad \forall t \in [T_i,T_{i+1}]$$ 
Note that $\displaystyle\frac{\partial J}{\partial x}$ may not always exist on the trajectory, \eg when it is along the nonsmooth curve. But $\displaystyle\frac{dJ}{dt}$ exists because a directional derivative exists. For all admissible trajectory $x(t)$ and corresponding $u(t)$, we do integration from $0$ to $\infty$:
because $\displaystyle\frac{dJ}{dt}$ exists everywhere on $[T_i,T_{i+1}]$ and it is Lebesgue integrable (easy to check for $C^1$ value function $J$), then using the Newton-Leibniz equation, we get:
$$
J(x(T_{i+1})) - J(x(T_i)) = \int_{T_i}^{T_{i+1}}\frac{dJ}{dt}dt
$$

$$\begin{aligned}\int _{0}^{T}\dfrac{dJ}{dt}dt&=\sum _{i}\int _{T_{i}}^{T_{i}+1}\dfrac{dJ}{dt}dt\\
 &=\sum _{i}J| _{T_{i}^+}^{T_{i+1}^-}\\
&=J\left( x\left( T\right) \right) -J\left( x\left( 0\right) \right) \end{aligned}$$
Let $t \rightarrow \infty$ we have 
$$\begin{aligned}J\left( x\left( +\infty \right) \right) -J\left( x\left( 0\right) \right)& =\int ^{+\infty }_{0}\dfrac{dJ}{dt}dt\\
&\geq -\int _{0}^{+\infty }c\left( x\left( t\right) ,u\left( t\right) \right) dt\end{aligned}$$
$$i.e., \quad J(x_0) \leq \int _{0}^{+\infty }c\left( x\left( t\right) ,u\left( t\right) \right) dt.$$
This shows that $J$ is a lower bound to the optimal value function $J^*$. Therefore, if $J$ can actually be attained by some admissible controller, then the optimality of $J$ is verified. Condition (vii) requires $J$ is indeed attainable, which is easy to realize from our numerical method in Section~\ref{sec:approach}. 
\end{proof}
\section{Proof of Theorem~\ref{the:suboptimality}}
\label{sec:app:proof-suboptimality}
\begin{proof}
The proof is similar to that of Theorem~\ref{the:optimality}, but because we assume $J$ is $C^1$ we need not split trajectories into multiple intervals.

We use the notation $u^*_{opt}(t)$ as the true optimal control that we never know, and $u^*_{sopt}(t)$ as the suboptimal control we calculate from $J(x)$.

If we integrate \eqref{eq:subop} for an arbitrary $u(t)$, $t \in [0,T_x]$ we have
$$\begin{aligned}J\left( x\left(T_x \right) \right) -J\left( x\left( 0\right) \right)& =\int ^{T_x }_{0}\dfrac{dJ}{dt}dt\\
&\geq -\left (\int _{0}^{T_x}c\left( x\left( t\right) ,u\left( t\right) \right) - l\left( x\left( t\right)\right) dt\right )\end{aligned}$$
$$\ie \quad J(x_0)- J(x_0(T_x)) \leq \int _{0}^{T_x }c\left( x\left( t\right) ,u\left( t\right) \right) dt - \int _{0}^{T_x }l\left( x\left( t\right)\right) dt$$
plug in the optimal policy $u^*_{opt}(t)$ we get 
$$
\quad J(x_0) - J(x_0(T_x)) \leq J^*(x_0)- J^*(x_0(T_x))  + \left\vert \int _{0}^{T_x }l\left( x^*\left( t\right)\right) dt \right\vert
$$
Rearranging terms we have 
\bea
    J(x_0)- J^*(x_0)  \leq   \left\vert \int _{0}^{T_x }l\left( x^*\left( t\right)\right) dt \right\vert + J(x_0(T_x)) - J^*(x_0(T_x)) \nonumber \\
    = \left\vert \int_{0}^{T_x} l(x^*(t)) dt \right\vert + \underbrace{J(x_0(T_x)) - J_\infty(x_0(T_x))}_{\leq \delta} + 
    \underbrace{J_\infty(x_0(T_x)) - J^*(x_0(T_x))}_{\leq J_\infty(x_0(T_x))\leq \varepsilon} \nonumber \\
    \leq \left\vert \int _{0}^{T_x }l\left( x^*\left( t\right)\right) dt \right\vert + \varepsilon + \delta \nonumber
\eea
We can use the mean value theorem of integral to bound the first term. Then,
$$0\leq J(x_0)- J^*(x_0)  \leq   \epsilon T_x + \delta + \varepsilon,
$$
which concludes the proof.
\end{proof}

\section{Sampling Algorithm}
\label{sec:app:sampling}

\begin{figure}[t]
	\vspace{-14mm}
	\begin{center}
		\begin{tabular}{cccc}
            \hspace{-5mm}	
            \begin{minipage}{0.14\textwidth}
				\centering
				\includegraphics[width=\textwidth]{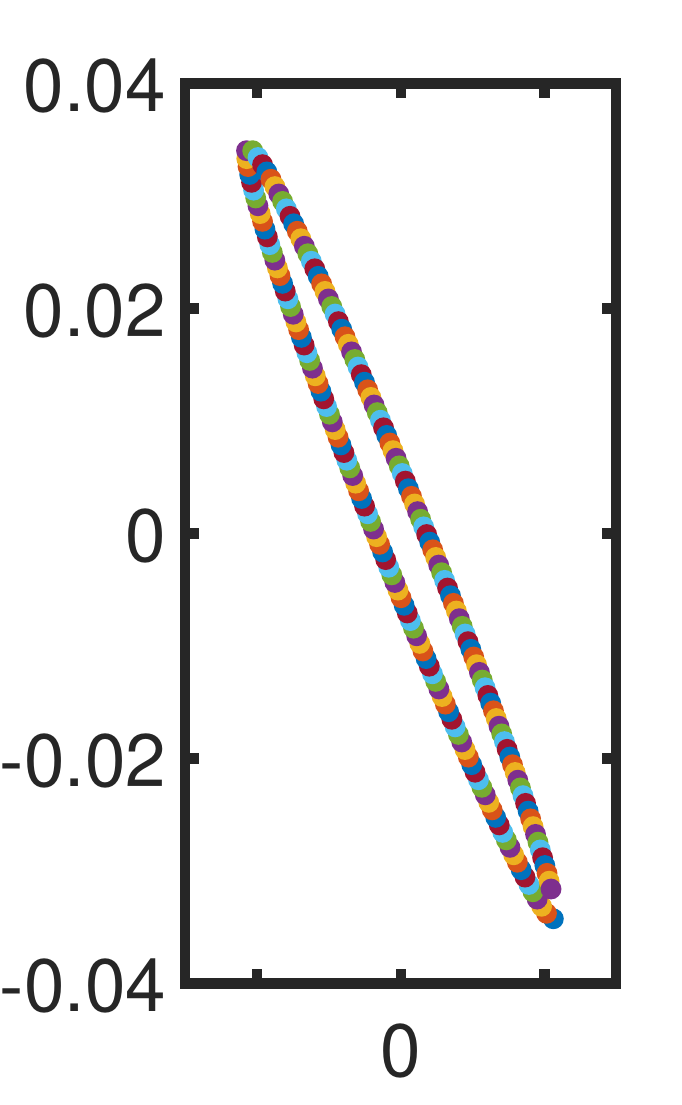}
			\end{minipage}
			&\hspace{-6mm}
			\begin{minipage}{0.38\textwidth}
				\centering
				\includegraphics[width=\textwidth]{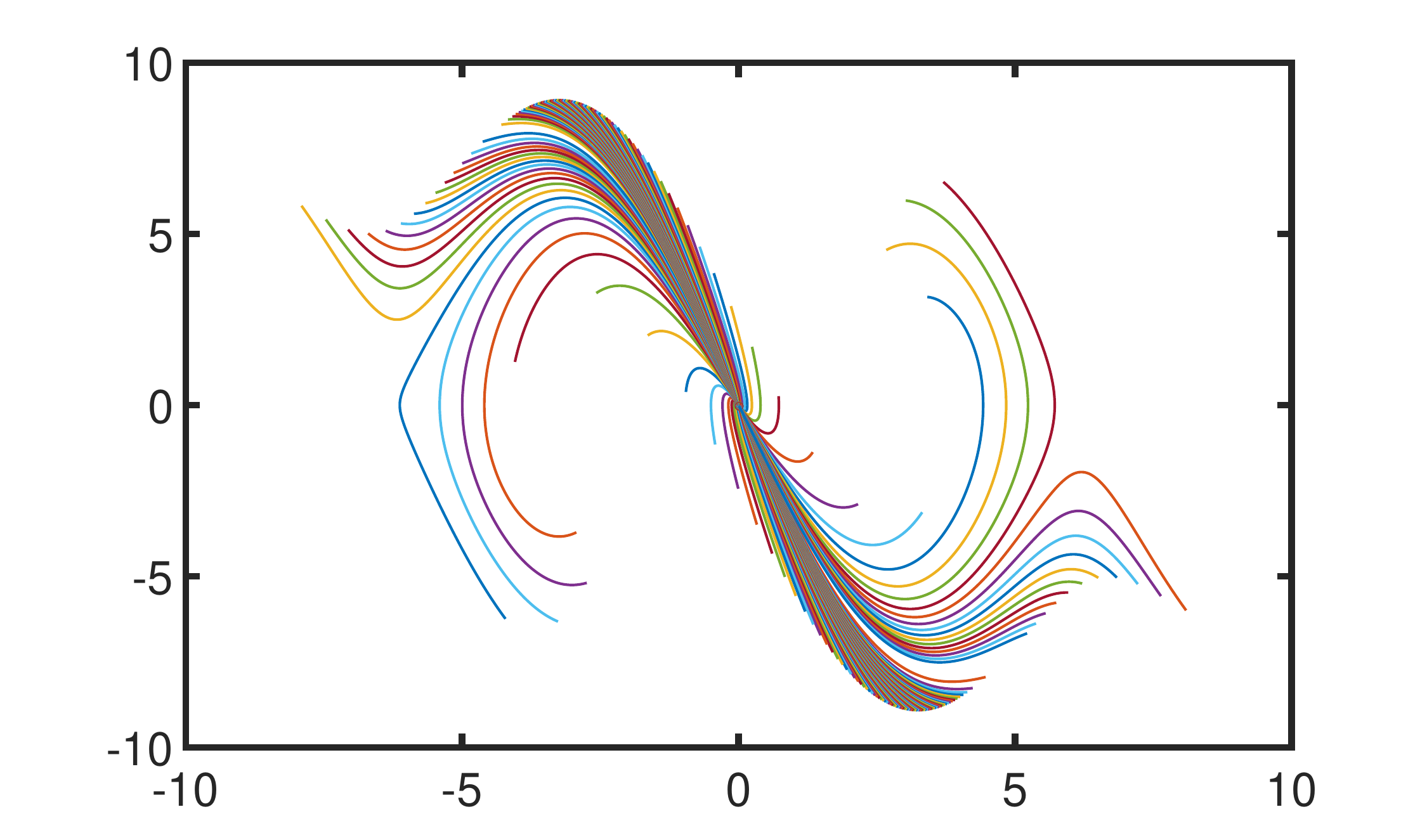}
			\end{minipage}
            &
            \hspace{-6mm}
			\begin{minipage}{0.14\textwidth}
				\centering
				\includegraphics[width=\textwidth]{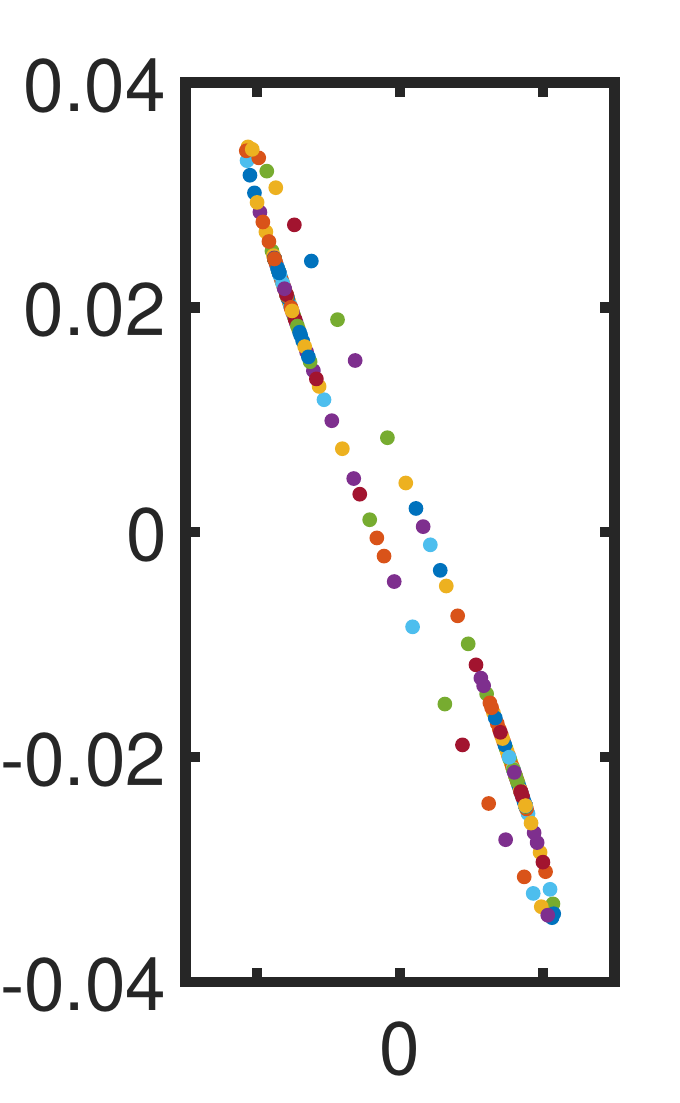}
			\end{minipage} 
			&
            \hspace{-6mm}
			\begin{minipage}{0.38\textwidth}
				\centering
				\includegraphics[width=\textwidth]{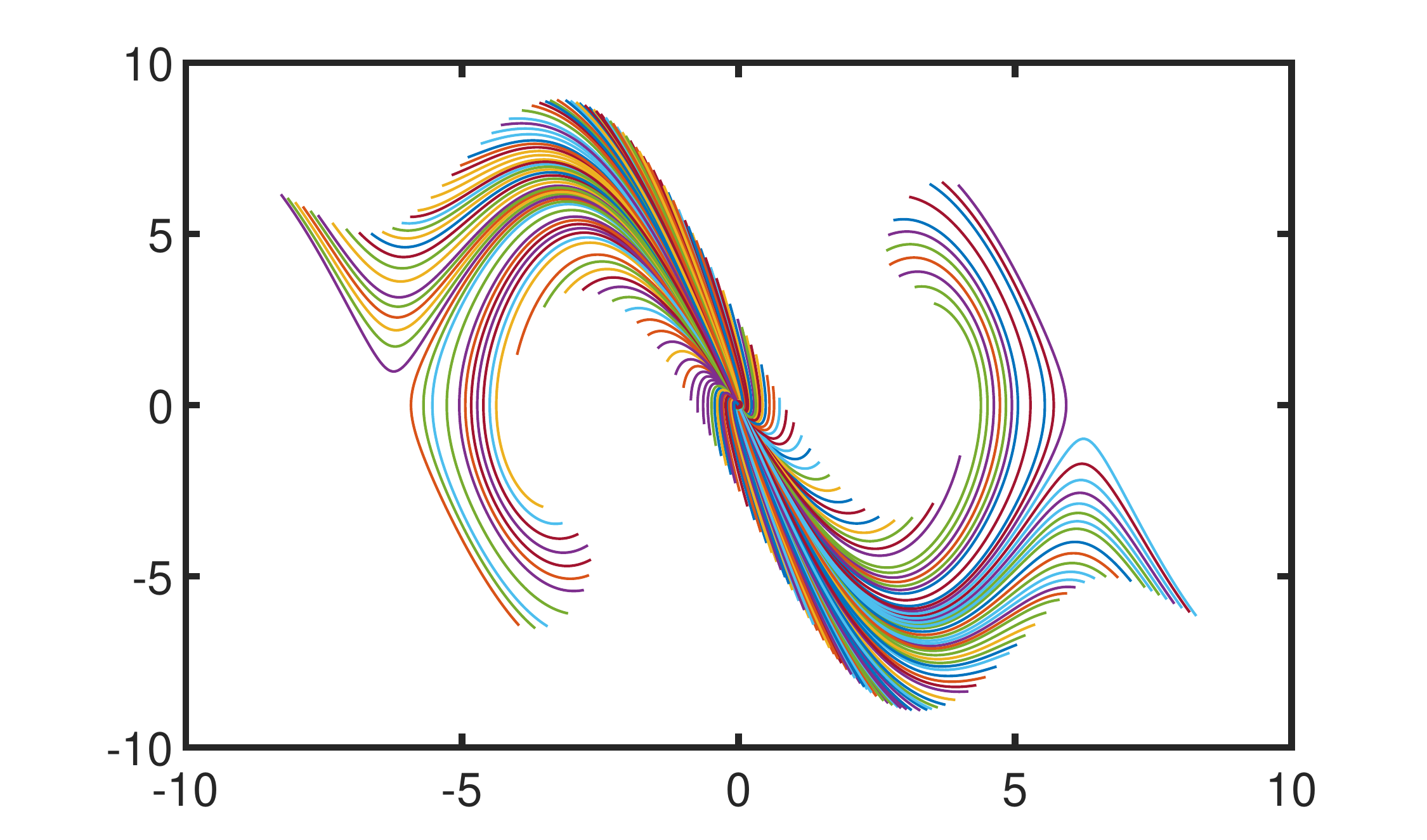}
			\end{minipage}\\
			\multicolumn{2}{c}{\smaller (a) uniformly distributed based on the perimeter of $\partial \calL$} & \multicolumn{2}{c}{\smaller (b) uniformly distributed based on $d(\cdot,\cdot)$}
		\end{tabular}
	\end{center}
	\vspace{-6mm}
	\caption{PMP trajectories resulting from two different sampling strategies. Left: based on the perimeter of the ellipse $\partial \calL$. Right: based on Algorithm~\ref{alg:nonsmooth-line}, initial states are not uniformly distributed on $\partial \calL$, but trajectories cover better the entire state space. Better viewed when zoomed in.\label{fig:sample}} 
	\vspace{-1mm}
\end{figure}

Fig.~\ref{fig:sample}(a) plots the PMP trajectories if we uniformly sample $\partial \calL$ based on its perimeter(recall $\partial \calL$ is the boundary of the LQR region as in~\eqref{eq:boundary-ellipse}). Observe that the PMP trajectories are clustered (and do not fully cover the state space) due to the LQR ellipse being highly elongated. This issue will get even more severe in high-dimensional systems because the condition number of the local LQR value function will get larger, \eg in a cart-pole system.



We present Algorithm~\ref{alg:sample} that performs sampling based on the distance metric $d(\cdot,\cdot)$ as in~\eqref{eq:distance-metric}. $N$ is the num of trajectories, $X$ is initial points on $\partial \calL$. Riccati() denotes solving the algebraic Riccati equation. Initial() denotes calculating the initial points based on $\theta$, $\theta$ is the parameter of $\partial \calL$. Levelset() denotes finding the $x^c$ defined in \eqref{eq:distance-metric}.

If the state space is 2D, $\mathcal{V} = \Real{2}$. If the space is $n$-D, we need to sample many subspaces to include in $\mathcal{V}$.


\setlength{\textfloatsep}{0pt}%
\begin{algorithm}[h]
    \SetAlgoLined
    \caption{Uniformly sample initial states based on $d(\cdot,\cdot)$\label{alg:sample}}
    \textbf{Input:} $N$,subspace set $\mathcal{V}$\\
    \textbf{Output:} $X$\\
    \For{$v$ in $\mathcal{V}$}{
        $P$ = Riccati($v$)\\
        $X$ = Initial($P$,$\{\pi,-\pi\}$)\\
        $L$ = Levelset($X$)\\
    \For{i = 1 \text{\textbf{to}} N}
    {$D = \emptyset $\\
    
    \For{j = 1 \text{\textbf{to}} i}{$D\{i\} = ||L\{i+1\}-L\{i\}||$}
    $j_c = \argmax_j D\{j\}$ \\
    $\theta_{\mathrm{new}} = \frac{\theta\{j_c\}+\theta\{j_c+1\}}{2}$\\
    $X_{\mathrm{new}} = \mathrm{Initial}($P$,\theta_{\mathrm{new}})$\\
    $L_{\mathrm{new}} = \mathrm{Levelset}(X_{\mathrm{new}})$\\
    Insert $L_{\mathrm{new}}$}
    }
\end{algorithm}

\section{Physical Explanation and Proof Related to Discontinuous Line}
\label{sec:app:discontinuous}
\subsection{The Singular Point}

Consider the point at which the maximum torque balances gravity, \ie $\thetadot = 0$ and $\theta$ satisfies
$$
mgl\sin\theta = u_\max.
$$
We call this point the equilibrium point (E-point). 

If $\thetadot=0$ and $mgl\sin\theta < u_\max$, then the pendulum can be directly driven to the upright point because we have enough torque.

If $\thetadot=0$ and $mgl\sin\theta > u_\max$, then the pendulum must swing down first to accumulate energy before swinging up to the upright position.

Therefore, the value function must be somewhat different at the E-point.


From the PMP trajectories shown in Fig.~\ref{app:fig:trajectories}, as $u_\max$ gets smaller, a singular point starts to appear. The singular point is very close to the E-point (under $10^{-3}$ error).

\subsection{Bang-bang Control and the Discontinuous Line}
\label{app:sec:bang-bang-discontinuous}
The singular point is actually on a ``discontinuous line''. Coincidentally, we observe that this line is the state trajectory of the bang-bang controller. 

At first glance, it may appear that discontinuity brings a problem in the proof in Appendix~\ref{sec:app:proof-optimality} because we need $J$ to be continuous. However, it is not strictly necessary. Recall the proof of Theorem~\ref{the:optimality}
\begin{subequations}
    \begin{align}
    \int _{0}^{T}\dfrac{dJ}{dt}dt&=\sum _{i}\int _{T_{i}}^{T_{i}+1}\dfrac{dJ}{dt}dt\\
    &=\sum _{i}J| _{T_{i}^+}^{T_{i+1}^-}\\
   &\leq J\left( x\left( T\right) \right) -J\left( x\left( 0\right) \right) \label{eq:discontinuous-new-proof}
\end{align}
\end{subequations}
   $t \rightarrow \infty$ we have 
   $$\begin{aligned}J\left( x\left( +\infty \right) \right) -J\left( x\left( 0\right) \right)& \geq\int ^{+\infty }_{0}\dfrac{dJ}{dt}dt\\
   &\geq -\int _{0}^{+\infty }c\left( x\left( t\right) ,u\left( t\right) \right) dt\end{aligned}$$
We only need \eqref{eq:discontinuous-new-proof} be true, \ie
$$
J_{T_i^-}\leq J_{T_i^+}.
$$
This can be explained as the trajectory will not jump from a high-value region to a low-value region. We give an intuitive proof.

\begin{wrapfigure}{r}{0.25\textwidth}
      \includegraphics[width = 0.25\textwidth]{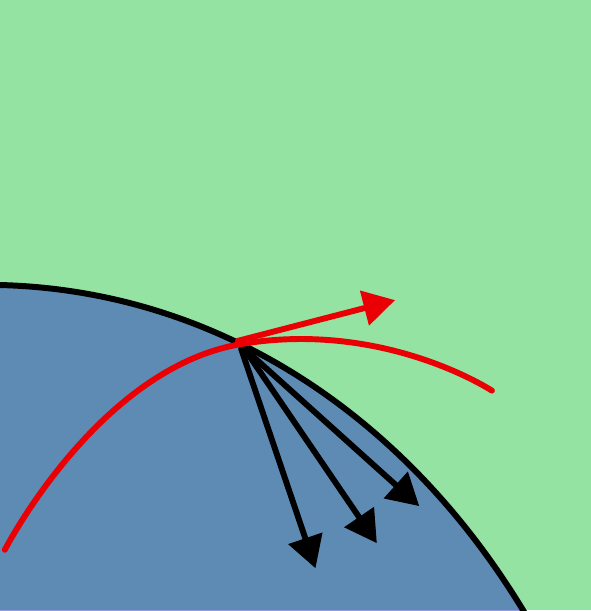}
      \vspace{-9mm}
      \caption{Trajectory\label{fig:dis-cross}}
      \vspace{-4mm}
   \end{wrapfigure}

\begin{proof}
   We only need to prove the trajectory can never cross the bang-bang control trajectory from the high-value region to the low-value region (blue and green regions in Fig.~\ref{fig:dis-cross}, respectively). The proof is straightforward: for control-constrained case, at one state $(\theta,\thetadot)$, not all the directions are admissible. Consider a special case $u_{\mathrm{max}} = 0$, \ie there is no control. Then the trajectory will only have one admissible direction. We will prove crossing the bang-bang control trajectory from the high-value region to the low-value region is not admissible.

   Without loss of generality, we consider a piece of bang-bang control trajectory in $\theta<0$, $\displaystyle\thetadot>0$ region, and the bang-bang controller is always applying $u_\max$, \ie the black line in Fig.~\ref{fig:dis-cross}.

   For bang-bang control, we have\beq
   \begin{bmatrix}
      \thetadot_b\\
      \ddot \theta_b
   \end{bmatrix}
   =
   \begin{bmatrix}
      \thetadot_b  \\
      -\frac{1}{ml^{2}} \left( b \thetadot_b - m g l \sin \theta_b - u_{\mathrm{max}} \right)
      \end{bmatrix}
   \eeq
   And for any other control, we have\beq
   \begin{bmatrix}
      \thetadot_a\\
      \ddot \theta_a
   \end{bmatrix}
   =
   \begin{bmatrix}
      \thetadot_a  \\
      -\frac{1}{ml^{2}} \left( b \thetadot_a - m g l \sin \theta_a - u_{\mathrm{a}} \right)
      \end{bmatrix}
   \eeq
   For a parametric curve $(\theta(t),\thetadot(t))$, the slope of its tangent line can be calculated as $\frac{\ddot\theta(t)}{\thetadot(t)}$. If the trajectory crosses the blue region to the green region, we will have 
   $$
   \frac{\ddot\theta_b(t)}{\thetadot_b(t)} > \displaystyle\frac{\ddot\theta_a(t)}{\thetadot_a(t)}.
   $$
   However, from $u_{\mathrm{a}} < u_{\mathrm{max}}$, $\thetadot_a(t) = \thetadot_b(t)>0$, we get $\ddot\theta_a(t) > \ddot\theta_b(t)$ and $\frac{\ddot\theta_b(t)}{\thetadot_b(t)} \leq \frac{\ddot\theta_a(t)}{\thetadot_a(t)}$, which induces contradiction.
\end{proof}
Therefore, we can still use Theorem~\ref{the:optimality} when the value function is discontinuous.

\subsection{Discount Factor}
\label{app:sec:discount-factor}

\begin{figure}[t]
	\vspace{-14mm}
	\begin{center}
		\begin{tabular}{cc}
            \hspace{-5mm}	
            \begin{minipage}{0.5\textwidth}
				\centering
				\includegraphics[width=\textwidth]{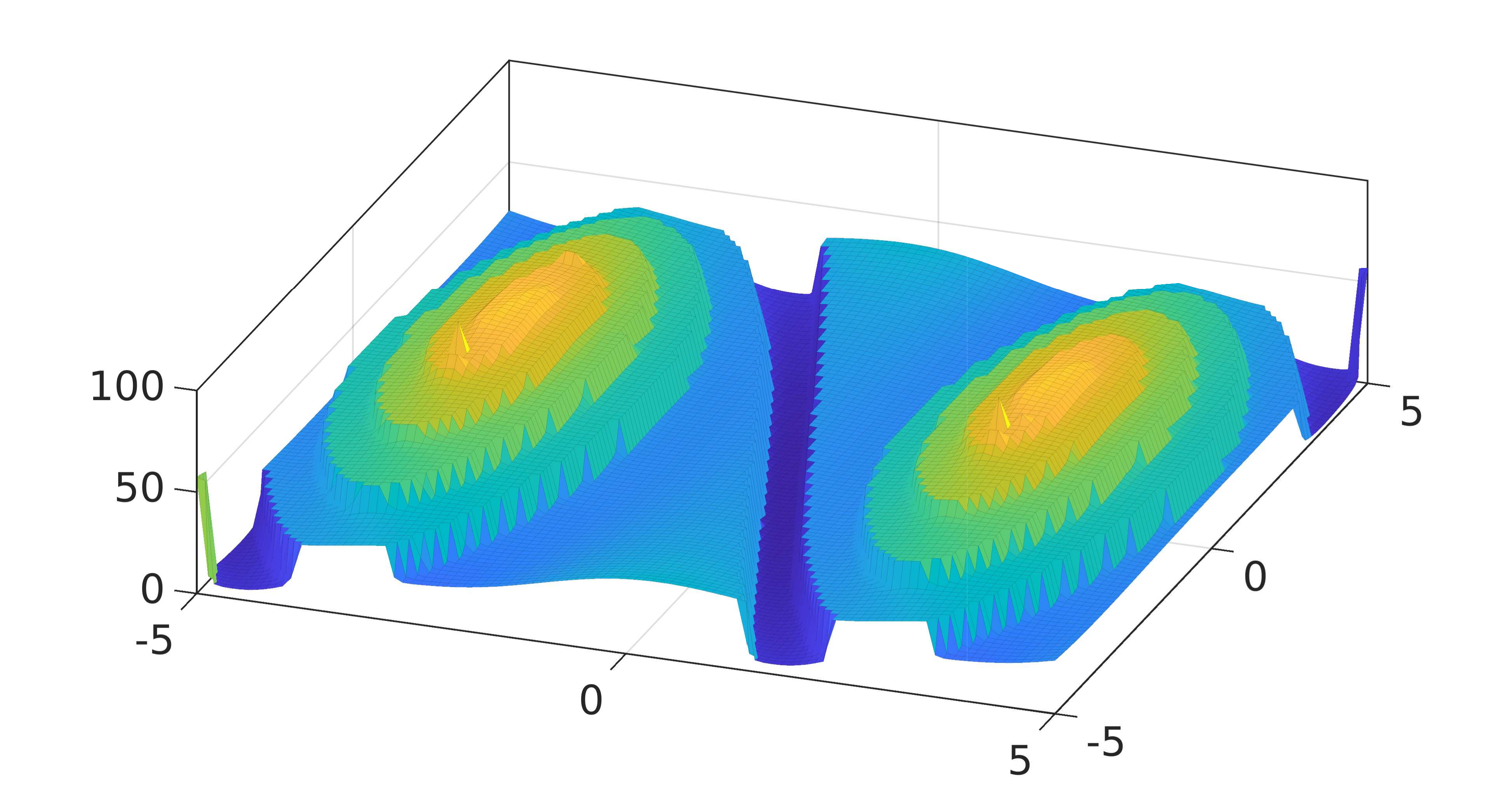}
			\end{minipage}
			&\hspace{-6mm}
			\begin{minipage}{0.5\textwidth}
				\centering
				\includegraphics[width=\textwidth]{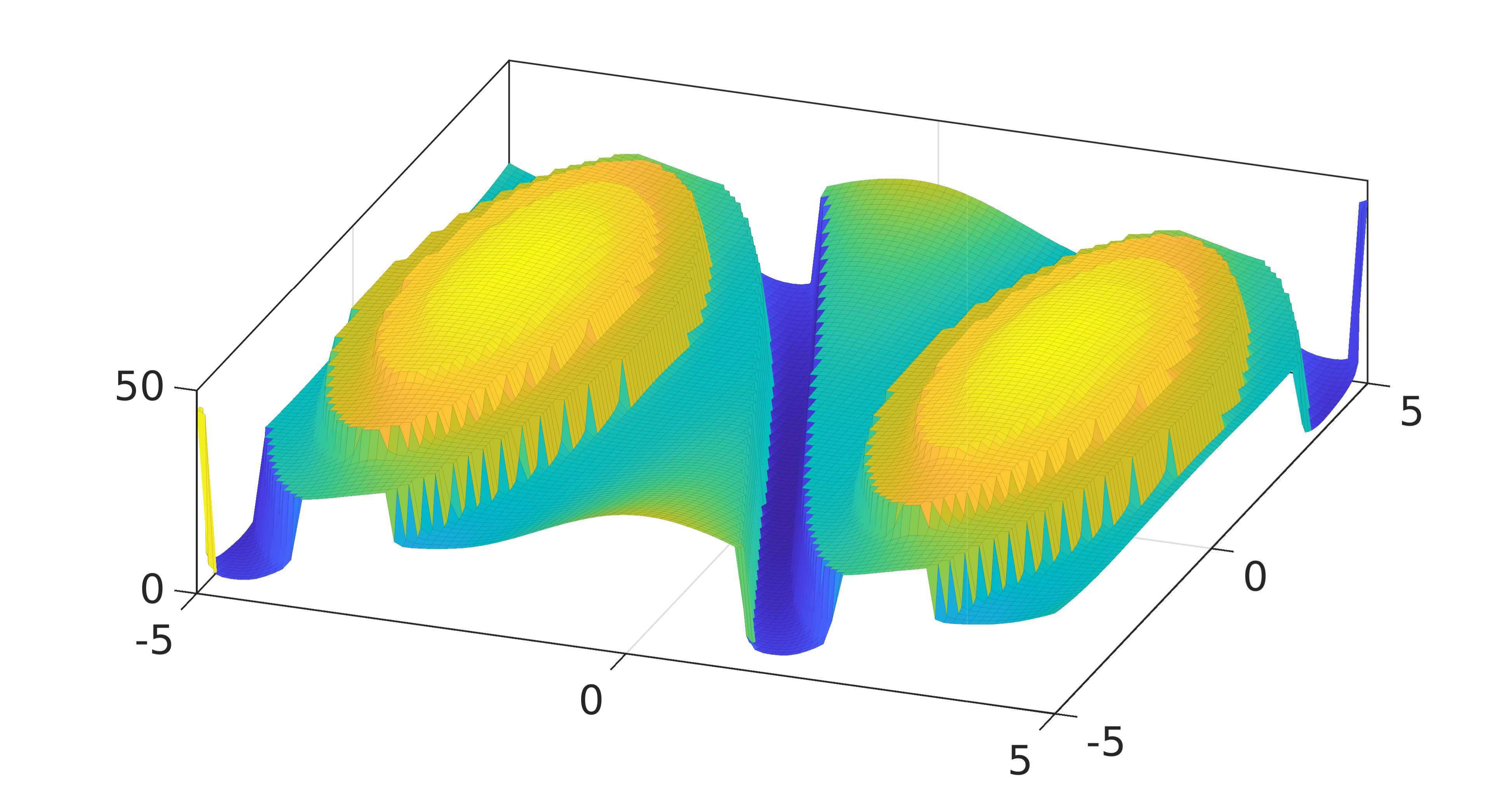}
			\end{minipage}\\
            (a) $\lambda = 0$ & (b) $\lambda = 0.15$
		\end{tabular}
		\begin{tabular}{cc}
            \hspace{-5mm}	
            \begin{minipage}{0.5\textwidth}
				\centering
				\includegraphics[width=\textwidth]{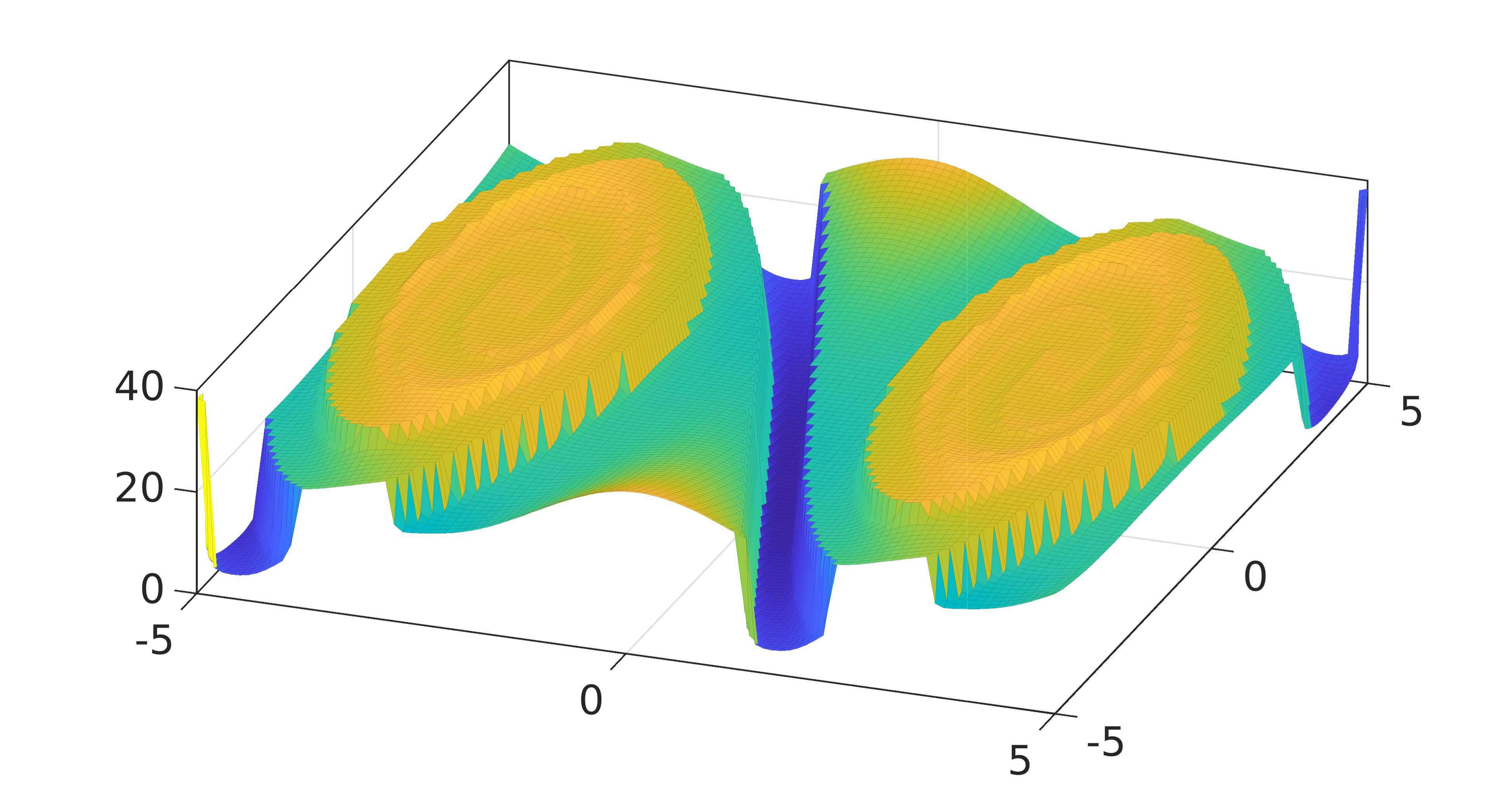}
			\end{minipage}
			&\hspace{-6mm}
			\begin{minipage}{0.5\textwidth}
				\centering
				\includegraphics[width=\textwidth]{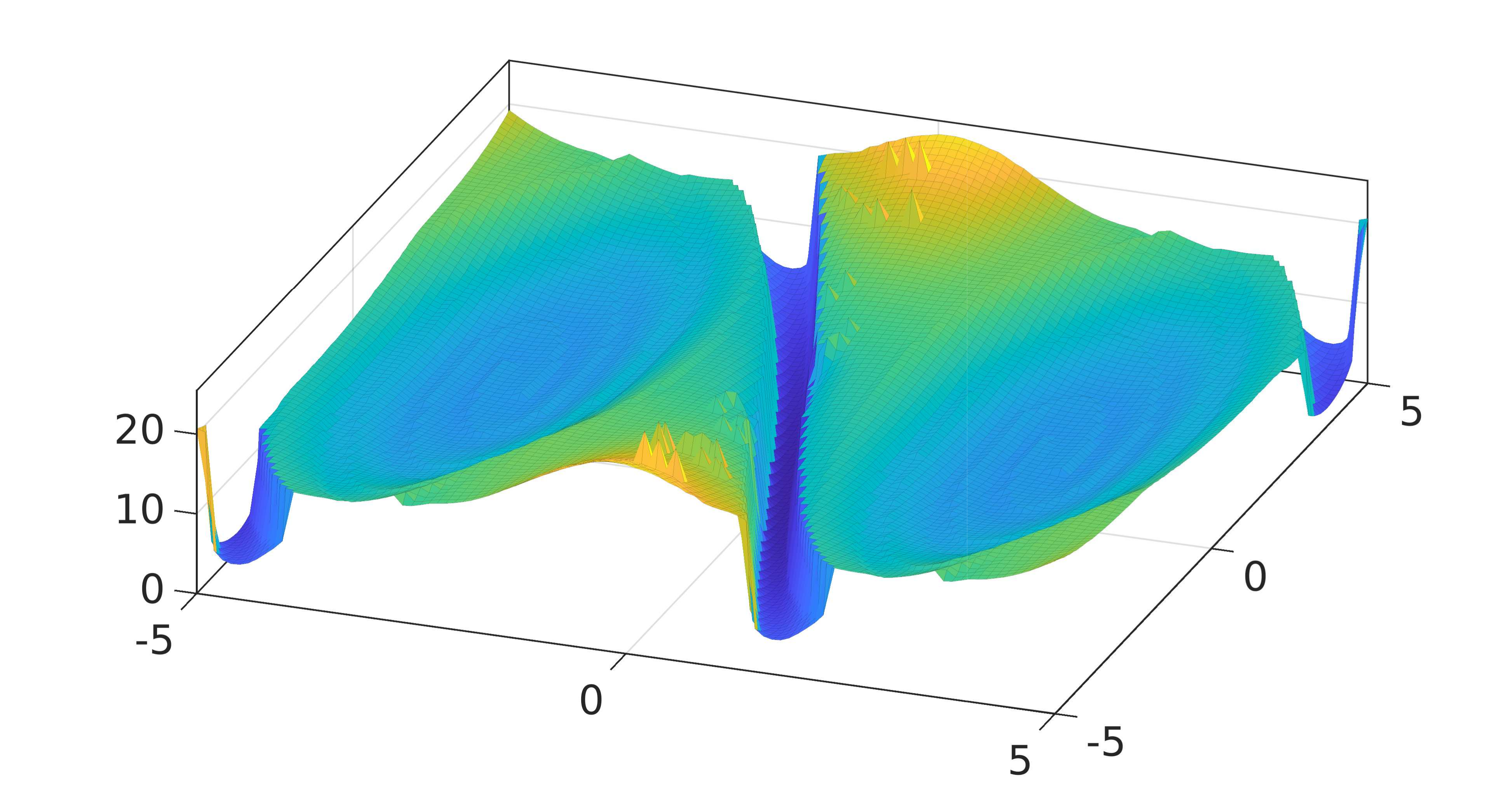}
			\end{minipage}\\
            (c) $\lambda = 0.3$ & (d) $\lambda = 1$
		\end{tabular}
	\end{center}
	\vspace{-6mm}
	\caption{Value functions for the discounted optimal control problem~\eqref{eq:discounted-optimal-control}. The value function appear to be discontinuous even with a discount factor. As the discount factor $\lambda$ grows larger, it is similar to the results from value iteration~\cite[Example 2.3]{yang23book-optimal}.\label{fig:discount}} 
	\vspace{-1mm}
\end{figure}

We attempted to prove the value function is continuous based on a classical result in~\cite[Theorem 1.5, P402]{bardi97book-optimal} that states if the optimal value function is continuous when there exists a discount factor, then the optimal value function without a discount factor is also continuous. 

Therefore, given the optimal control problem with a discount factor $\lambda > 0$
\bea\label{eq:discounted-optimal-control}
J(x_0) = \min _{u}\int ^{+\infty }_{0}e^{-\lambda t}c\left( x\left( t\right) ,u\left( t\right) \right) dt ,\quad    \dot x\left( t\right) =f\left( x\left( t\right) ,u(t)\right) ) ,\quad   x\left( 0\right) =x_{0},
\eea
we need to verify if the value function is Lipschitz continuous.
The corresponding HJB equation is 
$$
-\lambda J + \min_u{c(x,u)+\frac{\partial J}{\partial x}\tran  f(x,u}) = 0.
$$
Since PMP can be derived from the HJB equation using the method of characteristics, we still use PMP but modify it as 
\beq
\dot{p}(t) = -\lambda p-\nabla_x H(x^*(t),u^*(t),p(t)),\quad p(T) = \nabla_x J(x^*(T)).
\eeq
We obtain the discounted value functions in Fig.~\ref{fig:discount}, which appear to be still discontinuous.

\section{Extra Numerical Results}
\label{sec:app:extra-results}

\subsection{PMP Trajectories}
\label{app:sec:pmp-trajectories}
\begin{figure}[t]
	\vspace{-14mm}
	\begin{center}
		\begin{tabular}{cc}
            \hspace{-5mm}	
            \begin{minipage}{0.5\textwidth}
				\centering
				\includegraphics[width=\textwidth]{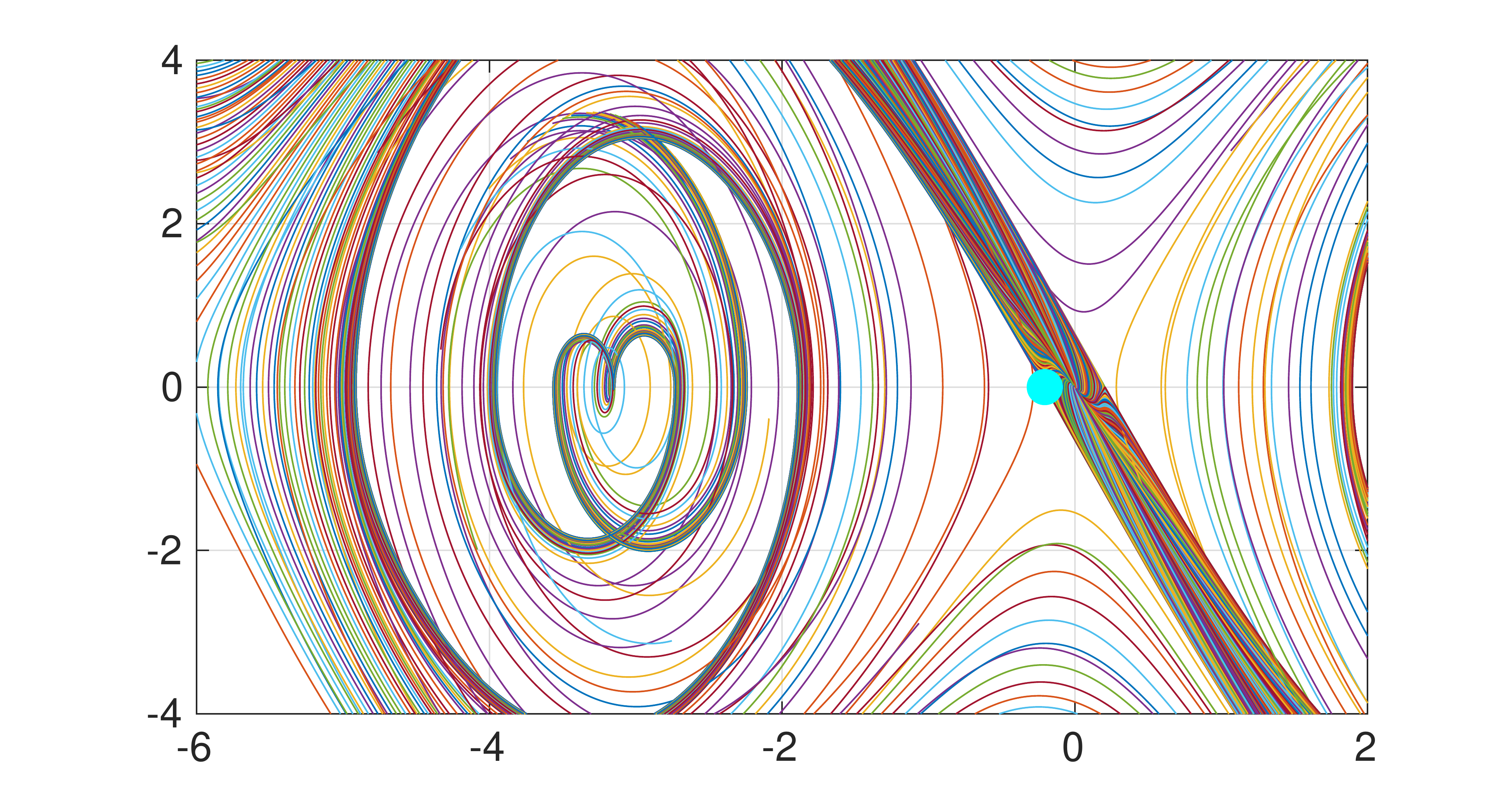}
			\end{minipage}
			&\hspace{-6mm}
			\begin{minipage}{0.5\textwidth}
				\centering
				\includegraphics[width=\textwidth]{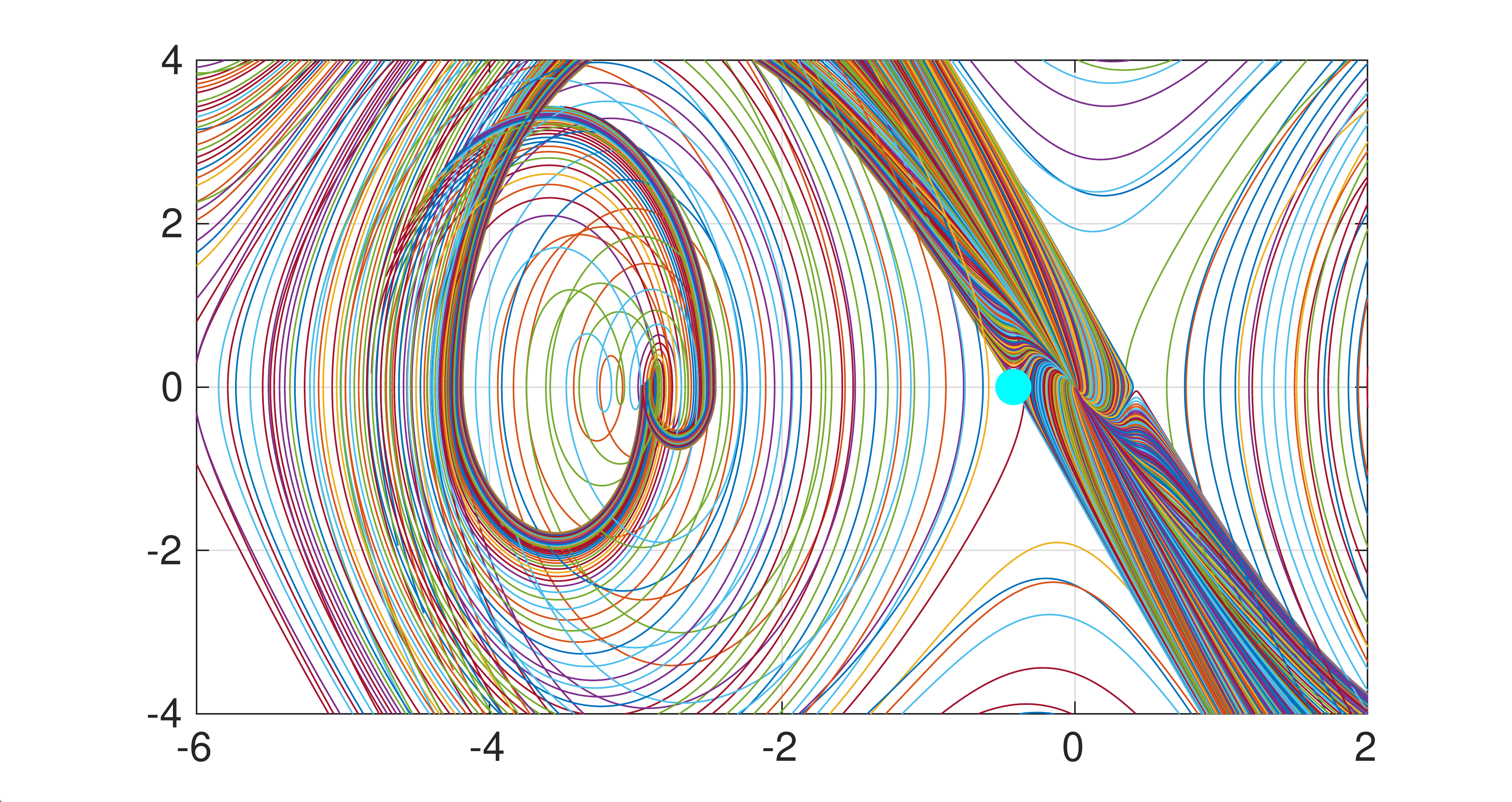}
			\end{minipage}\\
            (a) $u_{\mathrm{max}} = 2$&(b) $u_{\mathrm{max}} = 4$
		\end{tabular}
		\begin{tabular}{cc}
            \hspace{-5mm}	
            \begin{minipage}{0.5\textwidth}
				\centering
				\includegraphics[width=\textwidth]{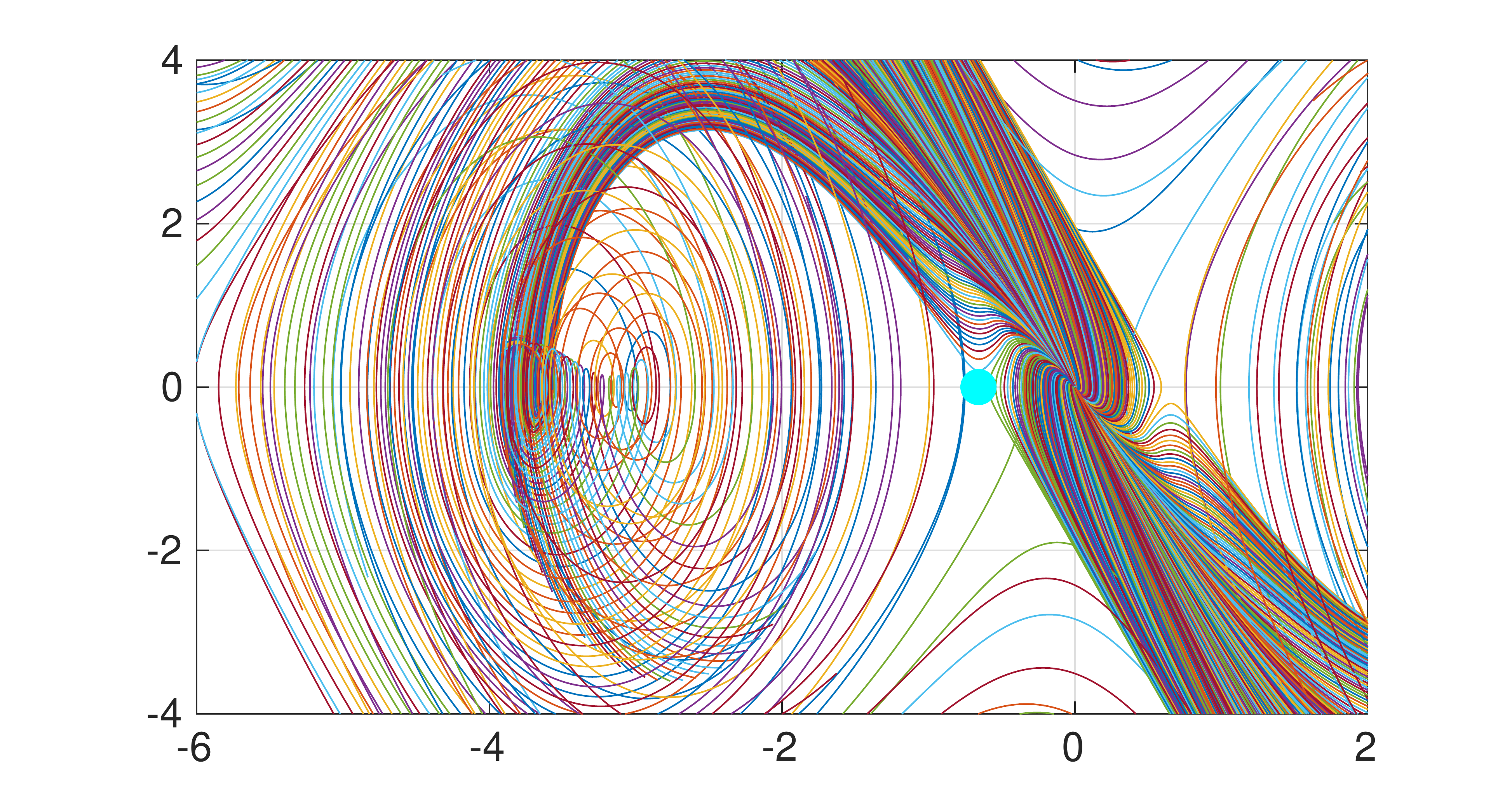}
			\end{minipage}
			&\hspace{-6mm}
			\begin{minipage}{0.5\textwidth}
				\centering
				\includegraphics[width=\textwidth]{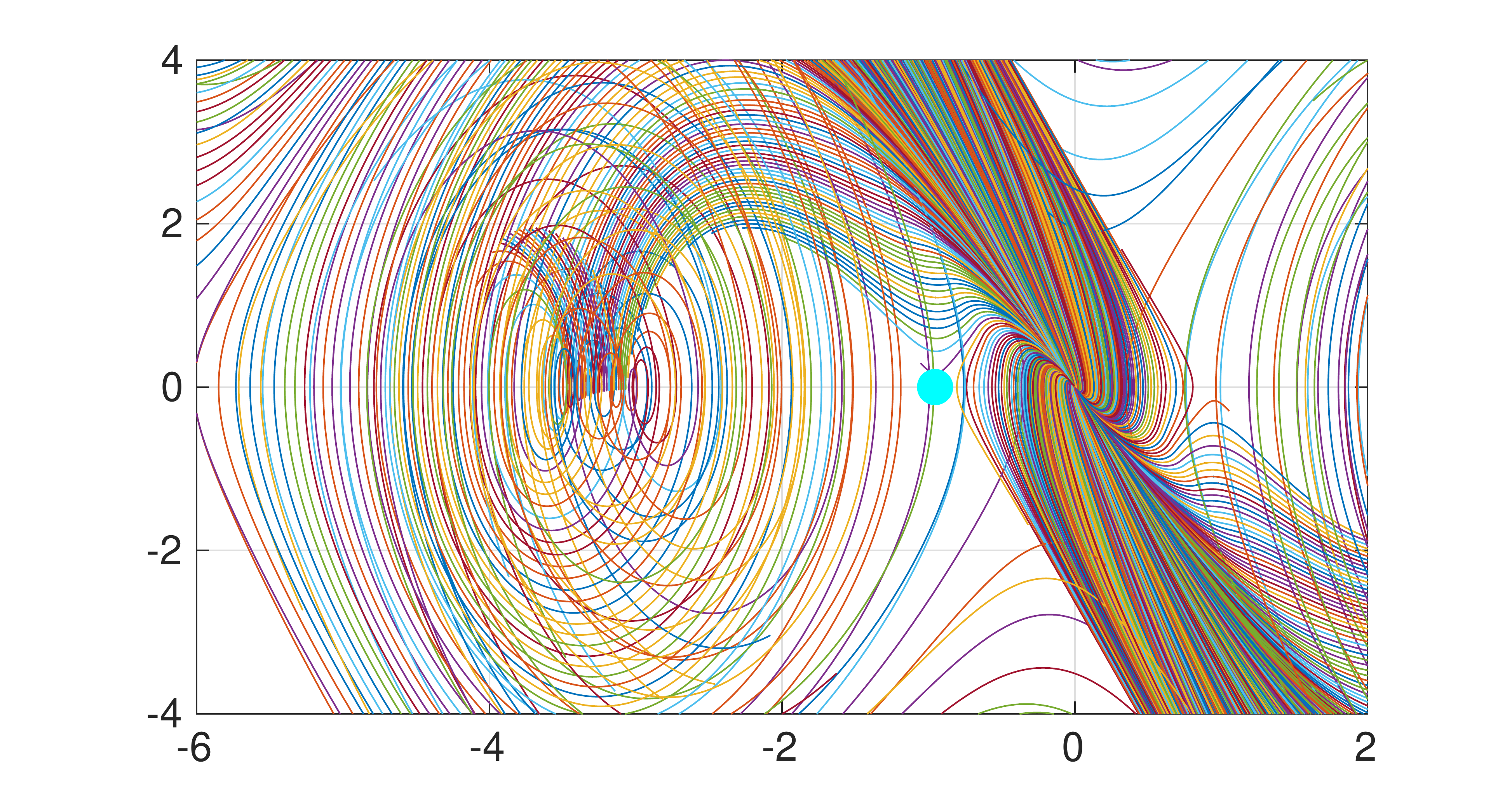}
			\end{minipage}\\
            (c) $u_{\mathrm{max}} = 6$&(d) $u_{\mathrm{max}} = 8$
		\end{tabular}
	\end{center}
	\vspace{-6mm}
	\caption{PMP trajectories with different values of $u_\max$. All the trajectories start near $(0,0)$ and span to different regions. The cyan point is the E-point $(-\arcsin\frac{u_{\mathrm{max}}}{mgl},0)$.\label{app:fig:trajectories}} 
	\vspace{-1mm}
\end{figure}

Fig.~\ref{app:fig:trajectories} plots the raw PMP trajectories. As we can see, they can intersect with each other and themselves. In fact, in the $(x,p)$ 4D state-costate space, the trajectories do not intersect. It is the projection of the trajectories from 4D to 2D that introduces the intersection.

\subsection{Contour Algorithm} 

In line~\ref{line:intersect} of Algorithm~\ref{alg:nonsmooth-line}, two contour lines may have multiple intersections. Given several initial intersection points, every time $k$ increases (the outer for loop), we choose the nearest intersection point of each initial point.
In terms of implementation, we form the contour line as a polygon and use the MATLAB intersect() function.

\subsection{HJB Residual}

\begin{figure}[t]
	\vspace{-14mm}
	\begin{center}
		\begin{tabular}{cc}
            \hspace{-6mm}	
            \begin{minipage}{0.5\textwidth}
				\centering
				\includegraphics[width=\textwidth]{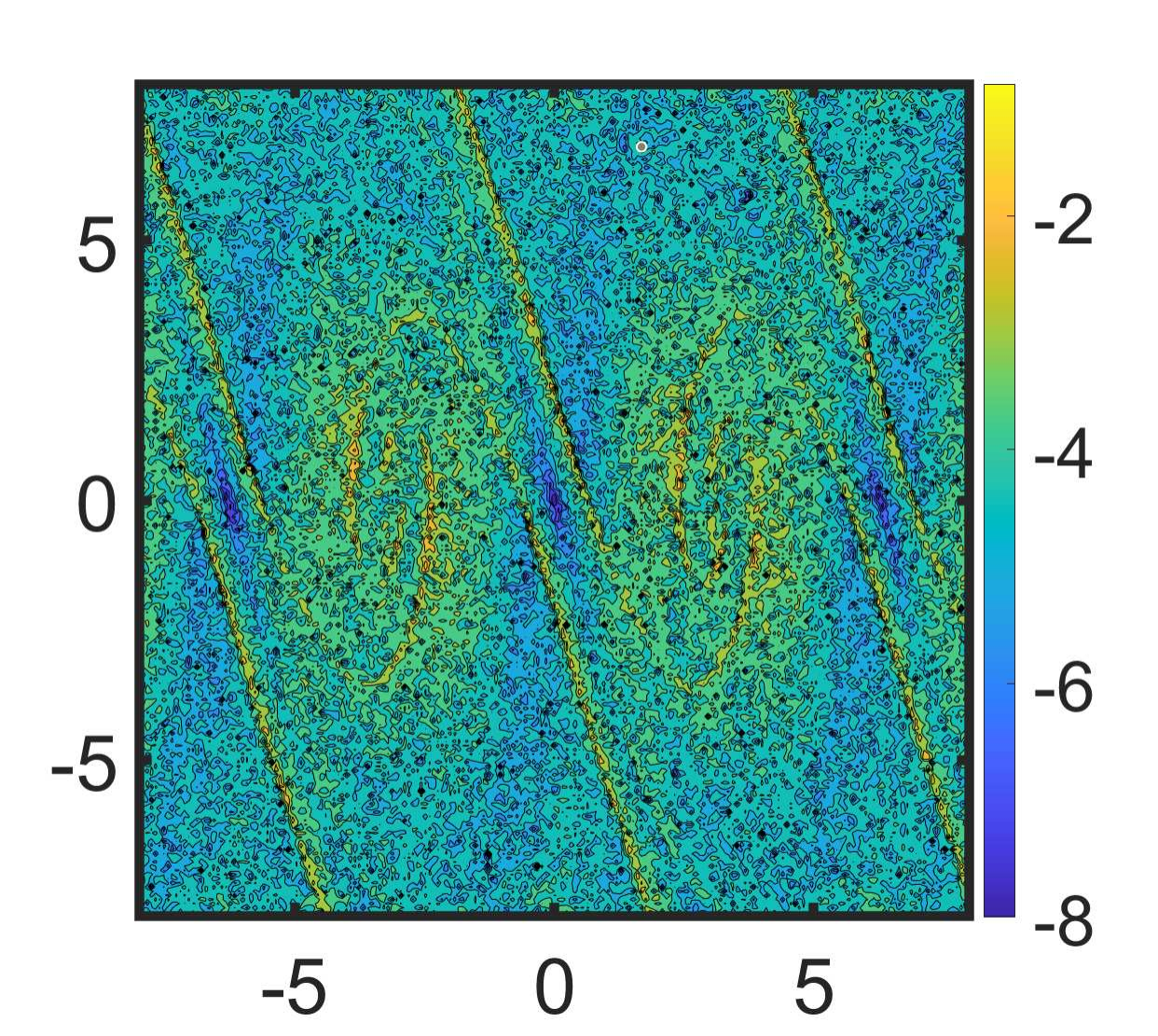}
			\end{minipage}
			&\hspace{0mm}
			\begin{minipage}{0.5\textwidth}
				\centering
				\includegraphics[width=\textwidth]{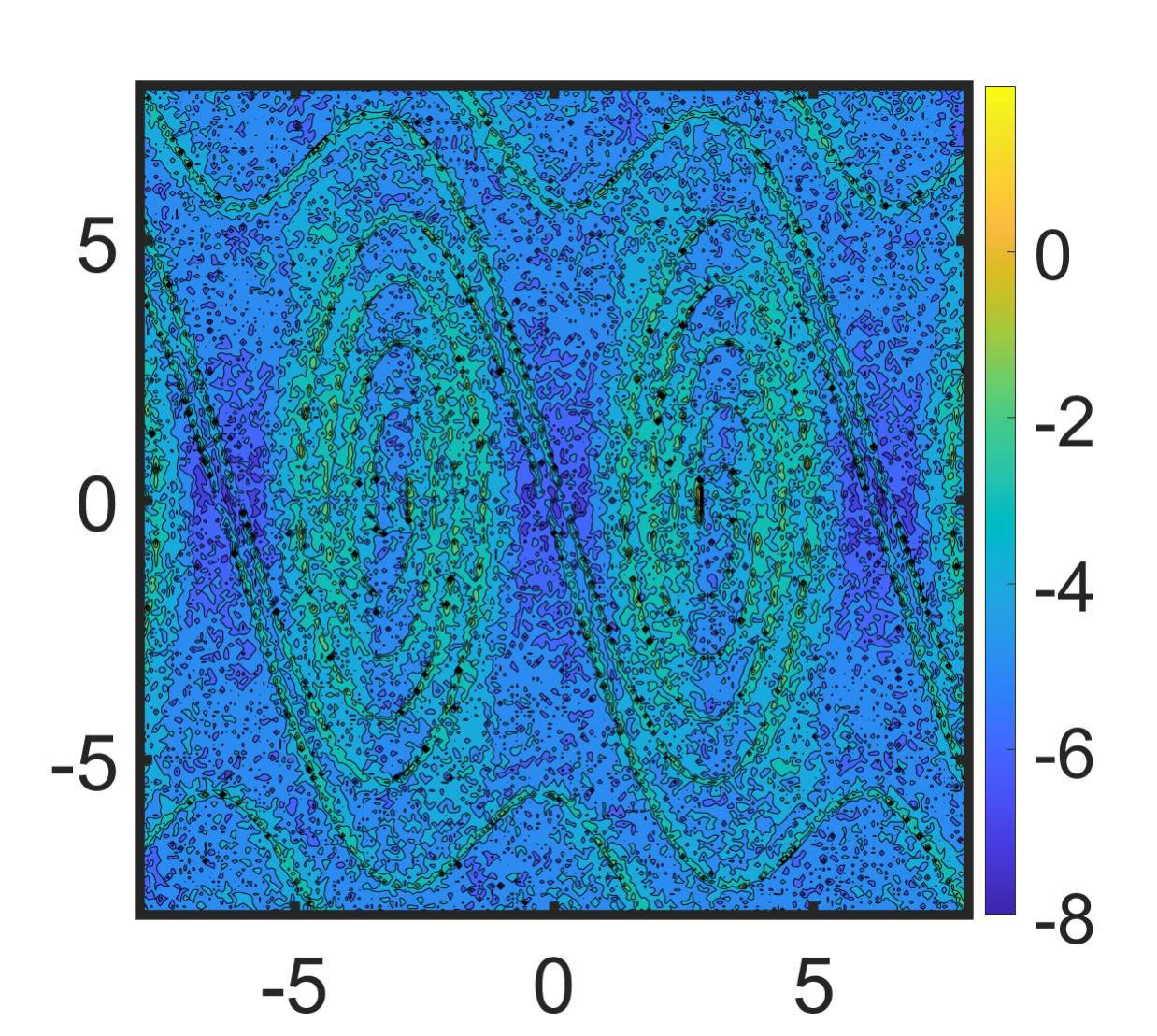}
			\end{minipage}\\
			(a) without control constraint & (b) with control constraint $u_\max = 2$
		\end{tabular}
	\end{center}
	\vspace{-4mm}
	\caption{$\log$-HJB residual.\label{fig:error}} 
	\vspace{6mm}

\end{figure}

Fig.~\ref{fig:error} shows the $\log$-HJB residual (\ie $-5$ denotes the HJB residual is $10^{-5}$). The average error is $1.3314\times 10^{-4}$ when there is no control saturation, and $1.2131\times 10^{-4}$ when $u_\max=2$. We can see the HJB residual near the nonsmooth curve is quite large.

\subsection{Neural Network Loss Functions}
\label{app:sec:nn-loss-functions}
Details about the neural network loss function in~\eqref{eq:nn-loss} are provided as follows
$$
\ell_{\lqr} = \frac{\displaystyle\sum^{N_\lqr}_{x_i \in \calL,i=1} (\Jnn(x_i) - J_{\infty}(x_i))^2}{N_{\lqr}}
$$
$$
\ell_{\Jvalue} = \frac{\displaystyle\sum_{i = 1}^{N_{\Jvalue}} (\Jnn(x_i) - J^*(x_i))^2}{N_{\Jvalue}}
$$
$$
\ell_{\hjb} = \frac{\displaystyle\sum_{i = 1}^{N_{\hjb}} (\mathrm{ReLU}(\min_{u \in \controlset}{c(x_i,u)+\displaystyle\frac{\partial \Jnn}{\partial x_i}^Tf(x_i,u)}))^2}{N_{\hjb}}
$$
$$
\ell_{\smooth} = \frac{\displaystyle\sum_{i = 1}^{N_{\smooth}} \left\Vert \displaystyle\frac{\partial \Jnn}{\partial x}(x_i) \right\Vert^2}{N_{\smooth}}.
$$
Here we choose $\lambda_{\lqr} = \lambda_{\Jvalue} = 1$, $\lambda_{\hjb} = \lambda_{\pmp} = 0.3$, $\lambda_{\smooth} = 0.002$. $N_{\lqr}= 100$, $N_{\hjb}$ and $N_{\smooth}$ are 10000 ($100 \times 100$ meshgrid data span from $[-7,7]\times[-7,7]$). We have  shown different $N_{\Jvalue}$ in Fig.~\ref{fig:neural}. The neural network structure is a 2-hidden layer 300 dimension Multi Layer Perceptron (MLP), with input $(\sin(\theta),\cos(\theta),\thetadot)$ and output $J$.

\subsection{Generalization to Cart-Pole}
\label{app:sec:cart-pole}
Consider the cart-pole dynamics shown in Fig.~\ref{fig:cart-pole}
\bea\label{eq:cartpole-dynamics}f(x(t),u(t)) = \begin{bmatrix}
\displaystyle\frac{1}{m_c+m_p\sin^2\theta}[u+m_p\sin\theta(l{\dot\theta}^2-g\cos\theta)]\\
\dot \theta  \nonumber \\
\displaystyle\frac{1}{l(m_c+m_p\sin^2\theta)}[-u\cos\theta-m_pl{\dot\theta}^2\sin\theta\cos\theta+(m_c+m_p)g\sin\theta]
\end{bmatrix},
\eea
where the state is $x = [\dot{s}, \theta, \thetadot]\tran$ with $s$ denoting the position of the cart. We consider $m_p = 0.2$ and $m_c = 1$ as the mass of the pendulum and the cart, respectively, $l = 1$ as the length of the pole and $g = 9.8$ as the gravity constant.
We consider the case without control saturation $u \in \Real{}$. 

The running cost is 
\bea\label{eq:cost-function-cartpole}
c( x, u ) = q_1 \sin^2 \theta + q_1 (\cos \theta - 1 )^2 + q_2 \thetadot^2 + q_3 \dot{s}^2 + ru^2,
\eea
\begin{wrapfigure}{r}{0.3\textwidth}
	\includegraphics[width = 0.3\textwidth]{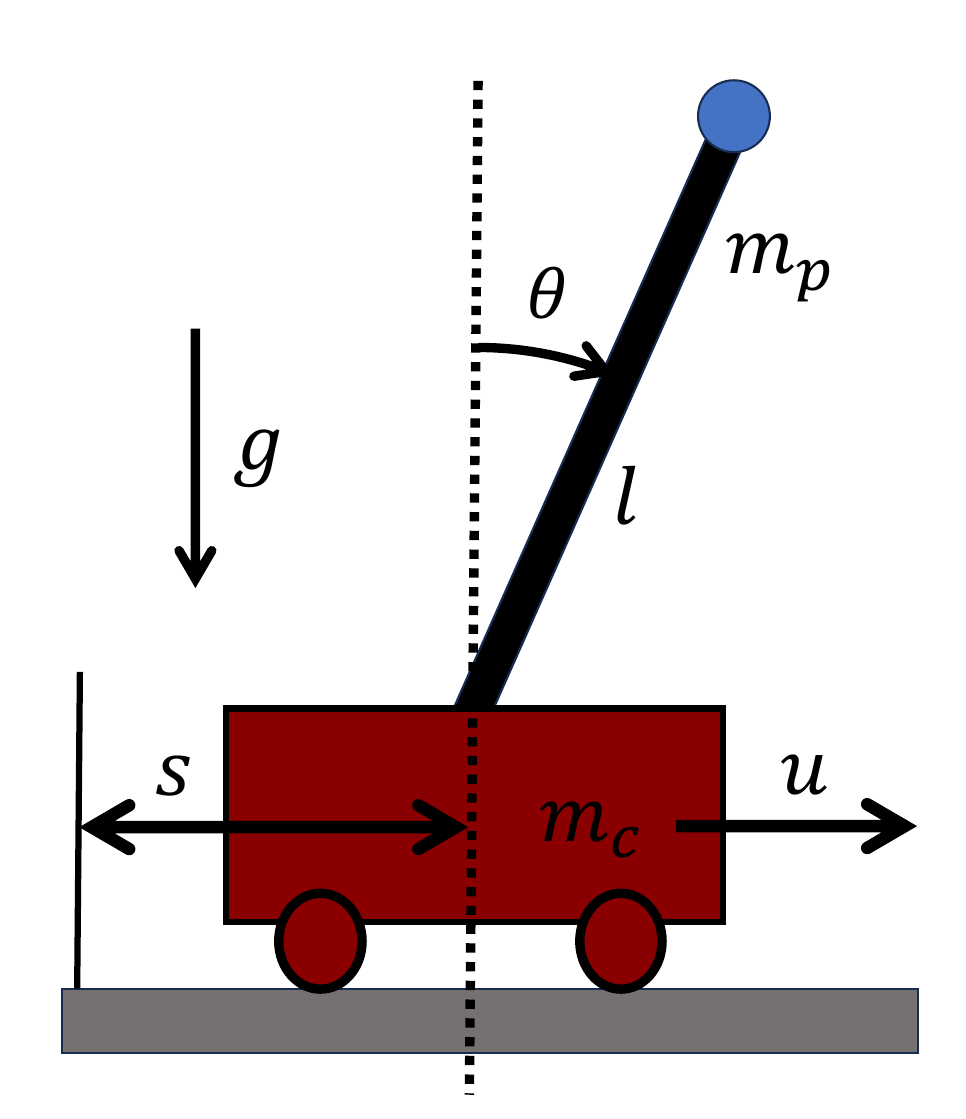}
	\vspace{-9mm}
	\caption{Cart-pole.\label{fig:cart-pole}}
	\vspace{-4mm}
\end{wrapfigure} 
with $q_1 = 1, q_2 = 1, q_3 = 1, r = 1$. We are interested in the region $x \in [-5,5]\times[-8,8]\times[-5,5]$  and the input to $\Jnn$ is $(\sin\theta, \cos\theta, \thetadot, \dot{s})$. We use a neural network with 2 hidden layers each
with 200 neurons, and the last layer is a RELU() layer in order to make the network value always positive.

We first use the approach introduced in Section~\ref{sec:approach} to generate 300 2-D subspaces, each with 1000 raw PMP trajectories. Then we follow the weak supervision training approach introduced in Section~\ref{sec:neural-approximation} to train a neural value function. Fig.~\ref{fig:cartpole} plots the neural value function and the state trajectories using the control induced by the neural value function.

\begin{figure}[t]
	\vspace{-14mm}
	\begin{center}
		\begin{tabular}{cc}
            \hspace{-6mm}	
            \begin{minipage}{0.4\textwidth}
				\centering
				\includegraphics[width=\textwidth]{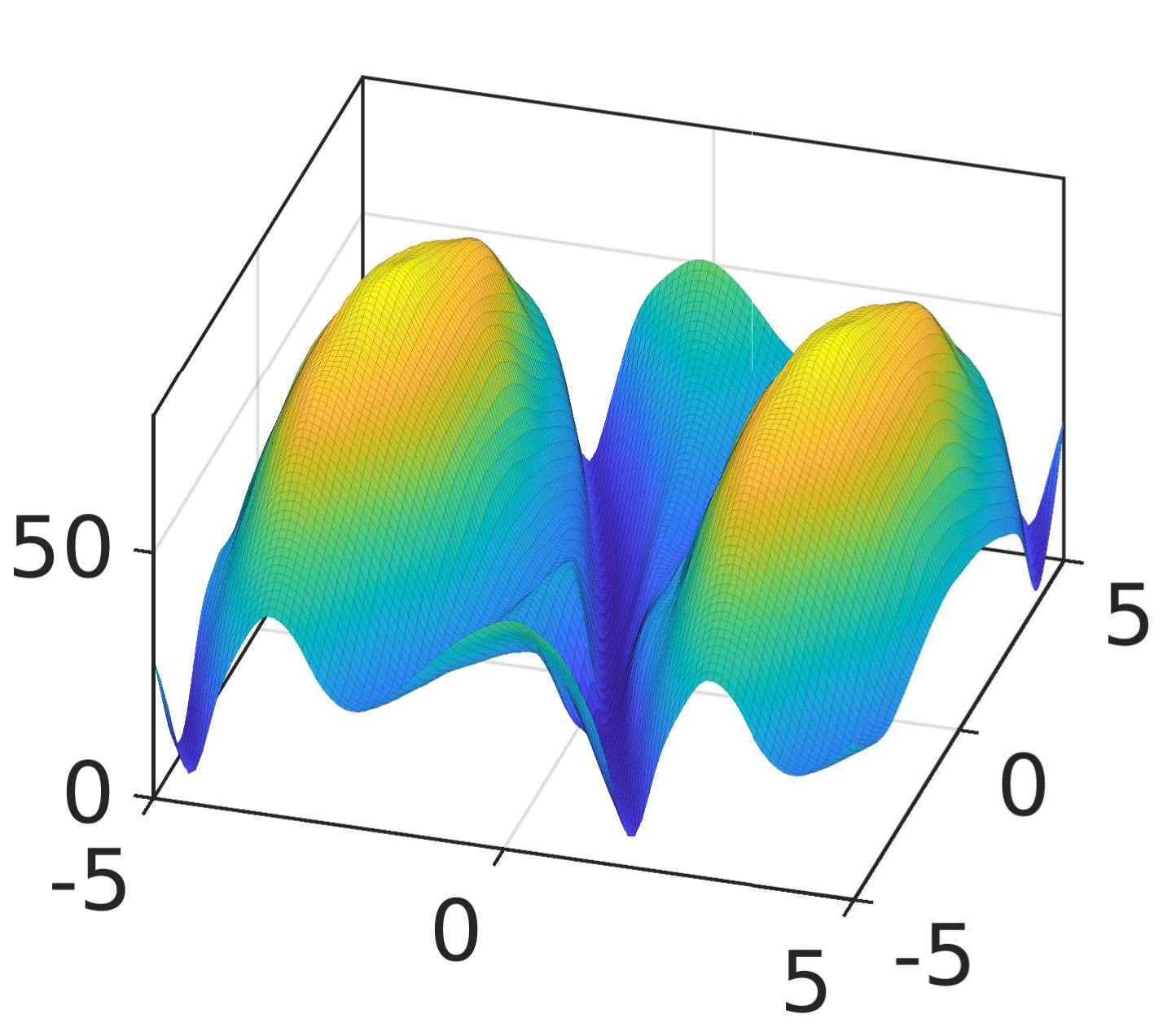}
			\end{minipage}
			&\hspace{0mm}
			\begin{minipage}{0.6\textwidth}
				\centering
				\includegraphics[width=\textwidth]{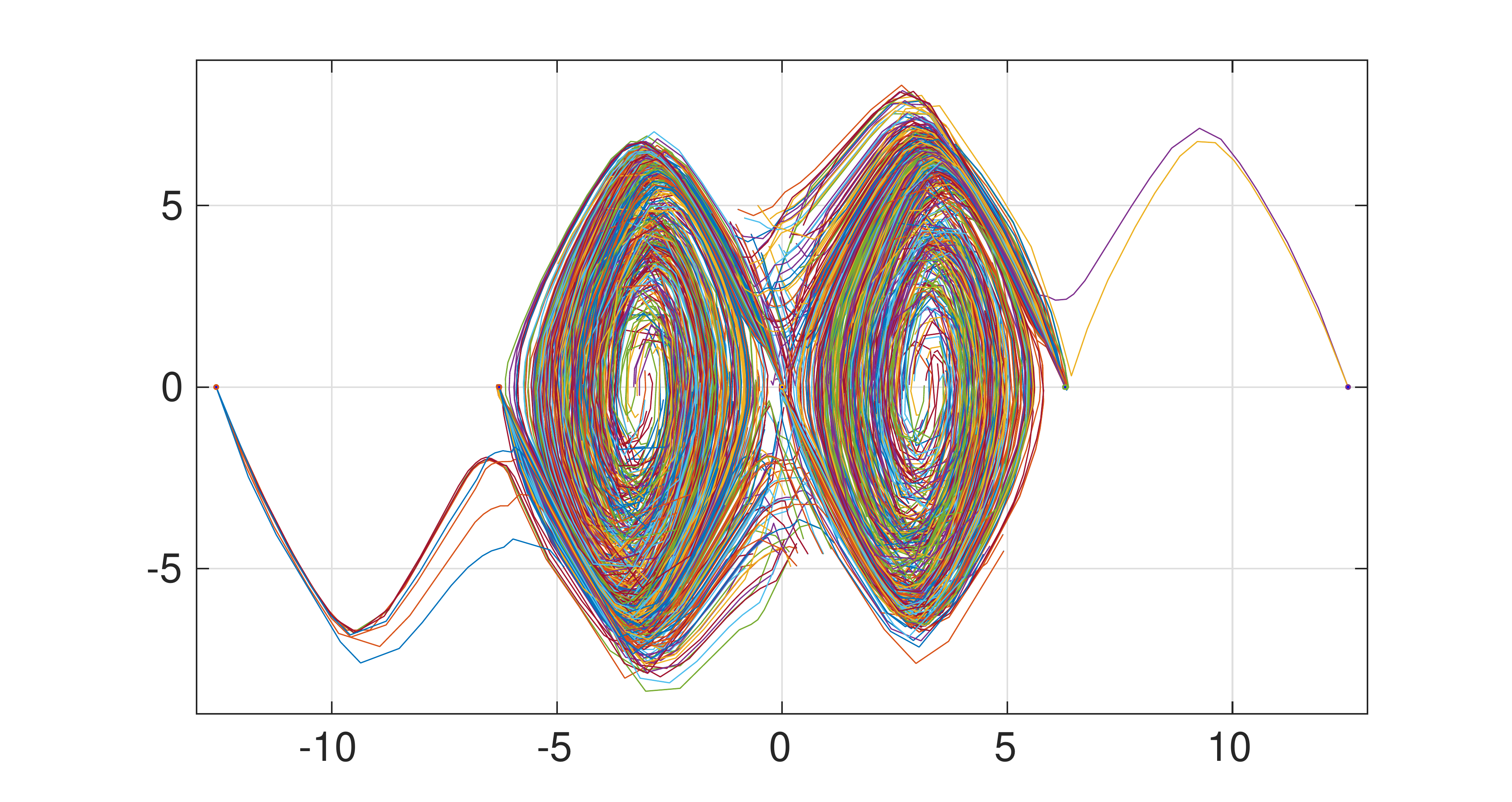}
			\end{minipage}\\
			(a) Neural value function & (b) State trajectories
		\end{tabular}
	\end{center}
	\vspace{-6mm}
	\caption{Neural value function and induced state trajectories for the cart-pole system. Both of them are in $\theta$ and $\thetadot$ coordinates. The value function is on the plane $\dot{s} = 0$. The 900 initial points are random samples in $[-5,5]\times[-5,5]\times[-5,5]$. \label{fig:cartpole}} 
	\vspace{2mm}
\end{figure}

\bibliography{reference}

\end{document}